\documentclass[reqno]{amsart}

\usepackage{amsmath,amssymb,amsthm,amscd}
\usepackage[mathscr]{eucal}
\usepackage{color}
\usepackage[all]{xy}

\newtheorem{theorem}{Theorem}
\newtheorem{lemma}[theorem]{Lemma}

\newtheorem{proposition}[theorem]{Proposition}

\theoremstyle{definition}

\newtheorem{example}[theorem]{Example}
\newtheorem{remark}[theorem]{Remark}

\newcommand{\C}{{\mathbb C}}

\newcommand{\Z}{{\mathbb Z}}

\newcommand{\B}{\mathcal{B}}
\newcommand{\cd}{\cdots}
\newcommand{\eb}{\bold{e}}

\renewcommand{\Im}{\mathrm{Im}\,}
\newcommand{\K}{\mathcal{K}}

\newcommand{\La}{\Lambda}

\newcommand{\ot}{\otimes}
\newcommand{\pair}[1]{\langle{#1}\rangle}
\newcommand{\tE}{\tilde{E}}
\newcommand{\tF}{\tilde{F}}
\newcommand{\U}{\mathbf{U}}
\newcommand{\vep}{\varepsilon}
\newcommand{\vphi}{\varphi}
\newcommand{\W}{\mathcal{W}}

\begin{document}

\title{Kirillov-Reshetikhin modules and quantum $K$-matrices 
}

\author{HIROTO KUSANO}
\address{Department of Mathematics, Osaka City University,
3-3-138 Sugimoto, Sumiyoshi-ku, Osaka, 558-8585, Japan}
\email{m21sa009@st.osaka-cu.ac.jp}

\author{MASATO OKADO}
\address{Osaka Central Advanced Mathematical Institute \&
Department of Mathematics, Osaka Metropolitan University,
3-3-138 Sugimoto, Sumiyoshi-ku, Osaka, 558-8585, Japan}
\email{okado@omu.ac.jp}

\author{HIDEYA WATANABE}
\address{Osaka Central Advanced Mathematical Institute \&
Department of Mathematics, Osaka Metropolitan University,
3-3-138 Sugimoto, Sumiyoshi-ku, Osaka, 558-8585, Japan}
\email{watanabehideya@gmail.com}


\begin{abstract}
From a quantum $K$-matrix of the fundamental representation, we construct one
for the Kirillov-Reshetikhin module by fusion construction. Using the $\imath$crystal
theory by the last author, we also obtain combinatorial $K$-matrices corresponding
to the symmetric tensor representations of affine type $A$ for all quasi-split Satake 
diagrams.
\end{abstract}

\maketitle

\section{Introduction}
A quantum symmetric pair, is a pair $(\U,\U^\imath)$ 
of a quantized enveloping algebra $\U$ and its coideal subalgebra $\U^\imath$.
The $\mathbf{U}^\imath$ itself is referred to as an $\imath$quantum group.
It is defined once we are given a Satake diagram (in the sense of \cite[Definition 2.7]{RV}) of symmetrizable Kac-Moody type.
Consider the Dynkin diagram of a symmetrizable Kac-Moody Lie algebra and let 
$I$ be the set of its nodes. A Satake diagram is a triple $(I,I_\bullet,\tau)$, where
$I_\bullet\subset I$ is a subdiagram of finite type and $\tau$ is a diagram automorphism
of order at most 2 satisfying certain conditions. From a Satake diagram of finite type (i.e., the Dynkin diagram $I$ is of finite type), Letzter constructed
a right coideal subalgebra $\U^\imath$ depending on parameters \cite{Le}.
Later, Kolb generalized her method to a wider class including arbitrary symmetrizable Kac-Moody types \cite{Ko}. Recently, 
new developments have been made. Universal $K$-matrix was obtained in \cite{BW1} (see also \cite{BK}),
and the theory of canonical bases for symmetric pairs, also known as $\imath$canonical
bases, was initiated in \cite{BW1} and developed in \cite{BW2,BW3}.

Let us go back to the usual quantum groups. We consider a quantum group $\U$
associated to an affine Lie algebra. It is known \cite{HKOTY:1999,HKOTT:2002,OS} 
that there exists a distinguished family of finite-dimensional $\U$-module, 
called Kirillov-Reshetikhin modules, that have crystal bases 
in the sense of Kashiwara \cite{Ka:1991}. For such Kirillov-Reshetikhin crystals,
a combinatorial version of the Yang-Baxter equation is satisfied by the combinatorial
$R$-matrix, and it is applied to the
analysis of box-ball systems, a kind of discrete integrable dynamical systems 
\cite{FOY,HKOTY:2002}.
We can also consider a box-ball system with boundary \cite{KOY}. In that case, together
with the Yang-Baxter equation, a combinatorial version of the reflection equation is also 
needed for integrability. It should be satisfied by the crystal limit of a quantum $K$-matrix.
However, we cannot take such crystal limit directly from the universal $K$-matrix.
In the case of the quantum $R$-matrix, the combinatorial $R$-matrix was determined
by using the crystal theory. Very recently, the theory of $\imath$crystal bases was
initiated by the last author in \cite{Wa2}, although it is restricted to quasi-split types 
$I_\bullet=\emptyset$. We would like to apply it to determine combinatorial $K$-matrices for
Kirillov-Reshetikhin crystals that satisfy the combinatorial reflection equation.

In this paper, we construct a quantum $K$-matrix for arbitrary Kirillov-Reshetikhin 
modules from that of fundamental representations. This method is well known as
fusion construction. 
However, we could only find \cite{BPO} for the case related to the quantum $K$-matrix, 
although it is not easy to see how their construction and ours are related.
So we decided to explain it in rather detail in this paper.
We then use the theory of $\imath$crystals in \cite{Wa2} 
and obtain explicitly
combinatorial $K$-matrices of the Kirillov-Reshetikhin crystals associated to the 
first fundamental representation, or symmetric tensor representations of arbitrary level,
for all Satake diagrams of quasi-split type based on the Dynkin diagram of affine type $A$. 
Satake diagrams of affine type are classified in \cite[Table 16]{RV}. According to it,
all quasi-split types are A.1, A.3 with no $\bullet$ and A.4. To be more precise, we 
also exclude A.3c to restrict our cases when $\tau(0)=0$ where 0 is the affine node.
All the combinatorial $K$-matrices in this paper are new, although the A.1 case
was also treated in \cite{KOYo}. This is due to the fact that the choices of parameters
appearing in $\U^\imath$ are different.
The parameters in this paper are chosen in a way such that the theory of $\imath$crystal is applicable, and the $\imath$crystal graph of the Kirillov-Reshetikhin crystals under consideration are connected.

Throughout the paper, we use the following notations: $\chi(st)=1\,(\text{$st$ is true}),
=0\,(\text{$st$ is false})$. For an integer $m$, we set 
$(m)_+=\max(m,0),\theta(m)=\chi(m\text{ is odd})$.

\section{Kirillov-Reshetikhin modules and fusion construction of quantum $K$-matrices}

\subsection{Quantum group}
We recall the definition of the quantum group associated with a symmetrizable generalized 
Cartan matrix $A=(a_{ij})_{i,j\in I}$. There exists a diagonal matrix $D=\mathrm{diag}(d_i)_{i\in I}$ 
such that $DA$ is symmetric. We take $d_i$'s to be pairwise coprime positive integers.
Let $\{\alpha_i\mid i\in I\}$ and $\{h_i\mid i\in I\}$ be the sets of simple roots and 
simple coroots. Then we have $\pair{h_i,\alpha_j}=a_{ij}$ for $i,j\in I$. The quantum group $\U$
associated to $A$ is defined to be an associative algebra over $\C(q)$ generated by
$E_i,F_i,K_i^{\pm1}$ ($i\in I$) subject to the relations
\begin{align*}
&K_iK_j=K_jK_i,\quad K_iE_j=q_i^{a_{ij}}E_jK_i,\quad K_iF_j=q_i^{-a_{ij}}F_jK_i,
\quad E_iF_j-F_jE_i=\delta_{ij}\frac{K_i-K_i^{-1}}{q_i-q_i^{-1}},\\
&\sum_{r=0}^{1-a_{ij}}(-1)^rE_i^{(r)}E_jE_i^{(1-a_{ij}-r)}
=\sum_{r=0}^{1-a_{ij}}(-1)^rF_i^{(r)}F_jF_i^{(1-a_{ij}-r)}=0\quad\text{if }i\ne j
\end{align*}
for $i,j\in I$, where $q_i=q^{d_i},[m]_i=\frac{q_i^m-q_i^{-m}}{q_i-q_i^{-1}},[m]_i!=\prod_{k=1}^m[k]_i,
E_i^{(m)}=E_i^m/[m]_i!,F_i^{(m)}=F_i^m/[m]_i!$.
$\U$ has the structure of a Hopf algebra with coproduct
\begin{equation} \label{coprod}
\Delta(E_i)=E_i\ot1+K_i\ot E_i,\quad \Delta(F_i)=F_i\ot K_i^{-1}+1\ot F_i,\quad \Delta(K_i)=K_i\ot K_i.
\end{equation}
This coproduct is sometimes denoted by $\Delta_+$ (see \cite{Ka:1991} for instance).

\subsection{$\imath$Quantum group}
We introduce an $\imath$quantum group $\U^\imath$. To do this, we first recall a
Satake diagram in the sense of \cite[Definition 2.7]{RV}. A Satake diagram is a pair of a Dynkin diagram whose vertices are
painted white or black and an involutive permutation $\tau$ on it satisfying 
$a_{\tau(i)\tau(j)}=a_{ij}$ for all $i,j\in I$. 
It is represented by a triple $(I,I_\bullet,\tau)$ where $I_\bullet(\subset I)$ is 
the set of black vertices. We also set $I_\circ=I\setminus I_\bullet$. Let $w_\bullet$
be the longest element of the Weyl group $W_\bullet$ associated to $I_\bullet$ and
$\rho^\vee_\bullet$ the half sum of the positive coroots of $I_\bullet$. Then 
$(I,I_\bullet,\tau)$ should satisfy 
\begin{align}
&w_\bullet(\alpha_i)=-\alpha_{\tau(i)}\quad\text{for }i\in I_\bullet,\\
&\pair{\rho^\vee_\bullet,\alpha_i}\in\Z\quad\text{for }i\in I_\circ\text{ such that }\tau(i)=i.
\end{align}

Next we recall 
Lusztig's braid group action $T_i=T''_{i,1}$ on $\U$ (\cite[37.1.3]{L}). It is defined by
\begin{align*}
T_i(E_j)&=\left\{\begin{array}{ll}
-F_iK_i&\text{if }i=j,\\
\sum_{r+s=-a_{ij}}(-1)^rq_i^{-r}E_i^{(s)}E_jE_i^{(r)}\qquad&\text{if }i\ne j,
\end{array}\right.\\
T_i(F_j)&=\left\{\begin{array}{ll}
-K_i^{-1}E_i&\text{if }i=j,\\
\sum_{r+s=-a_{ij}}(-1)^rq_i^rF_i^{(s)}F_jF_i^{(r)}\qquad&\text{if }i\ne j,
\end{array}\right.\\
T_i(K_j)&=K_jK_i^{-a_{ij}}.
\end{align*}
Set
\[
T_{w_\bullet}=T_{i_1}\cdots T_{i_l}
\]
when $w_\bullet=s_{i_1}\cdots s_{i_l}$ is a reduced expression.

To define a generator $B_i$, we introduce two sets of parameters
$(\varsigma_i)_{i\in I_\circ}\in(\C(q)^\times)^{I_\circ},(\kappa_i)_{i\in I_\circ}
\in\C(q)^{I_\circ}$. $\varsigma_i$ should satisfy $\varsigma_i=\varsigma_{\tau(i)}$ for
$i\in I_\circ$. We set
\begin{equation} \label{B_i}
B_i=F_i+\varsigma_iT_{w_\bullet}(E_{\tau(i)})K_i^{-1}+\kappa_iK_i^{-1}.
\end{equation}
The $\imath$quantum group $\U^\imath$ defined as a subalgebra of $\U$ generated
by $\U(I_\bullet)$ and $\{B_i,K_iK_{\tau(i)}^{-1}\}_{i\in I_\circ}$, where $\U(I_\bullet)$ is
a usual quantized enveloping algebra associated to $I_\bullet$. $\U^\imath$ is a right
coideal with respect to the coproduct given in the previous subsection.

\subsection{Kirillov-Reshetikhin modules} \label{subsec:KRmodule}
In this subsection, we assume $A$ is of affine type and $0\in I$ to be the affine node.
Let $\{\La_i\mid i\in I\}$ be the set of fundamental weights and set
$P=\bigoplus_{i\in I}\Z\La_i\oplus\Z\delta$ where $\delta$ is the standard null root.
Since we are interested in finite-dimensional $\U$-modules, our weight lattice to
consider should be $P_{\mathrm{cl}}=P/\Z\delta$. 
For a finite-dimensional $\U$-module $U$, let $U_{\mathrm{aff}}$ denote the 
$\U$-module $\C(q)[z,z^{-1}]\ot U$ with the actions of $E_i$ and $F_i$ by 
$z^{\delta_{i0}}\ot E_i$ and $z^{-\delta_{i0}}\ot F_i$. For $x\in\C(q)$, we define
the $\U$-module $U(x)$ by $U_{\mathrm{aff}}/(z-x)U_{\mathrm{aff}}$.
For finite-dimensional modules $U,V$ and $x,y\in\C(q)$, we introduce 
the quantum $R$-matrix $R_{U,V}(x,y):U(x)\ot V(y)\rightarrow V(y)\ot U(x)$
as a $\U$-linear operator. If $U(x)\ot V(y)$ is irreducible, $R_{U,V}(x,y)$
is an isomorphism and is determined up to scalar multiple. But it happens 
for some special elements $x,y$ of $\C(q)$ that $R_{U,V}(x,y)$ is not an isomorphism.
Quantum $R$-matrices satisfy the Yang-Baxter equation.
\begin{equation} \label{YBeq}
R_{V,W}(y,z)R_{U,W}(x,z)R_{U,V}(x,y)=R_{U,V}(x,y)R_{U,W}(x,z)R_{V,W}(y,z).
\end{equation}

We next explain Kirillov-Reshetikhin modules, KR modules for short, following \cite[\S3]{OS}. 
Set $I_0=I\setminus\{0\}$. 
For $r\in I_0$, we define $\varpi_r=\La_r-\pair{h_r,c}\La_0$ and call it a level 0 fundamental
weight. Here $c$ is the canonical central element. 
In \cite{Ka:2002}, a finite-dimensional $\U$-module $W(\varpi_r)$ is constructed as a quotient 
of the extremal weight module $V(\varpi_r)$. It is called a (level 0) fundamental module.
Let $\W$ be a fundamental module. From $\W$, we can construct a KR
module $\W_s$ ($s\in\Z_{>0}$) by fusion construction.
For $s\geq 2$, let $\mathfrak{S}_s$ denote the group of permutations on $s$ letters generated by $\sigma_i=(i\ i+1)$ for $1\leq i\leq s-1$.
We have $\U$-linear maps 
\begin{equation}\label{R_w}
R_w(x_1,\ldots,x_s): 
\W(x_1)\otimes \cdots\otimes \W(x_s) \longrightarrow \W(x_{w(1)})\otimes \cdots\otimes \W(x_{w(s)}),
\end{equation}
for $w\in \mathfrak{S}_s$ and $x_1,\ldots, x_s\in \C(q)$ satisfying
\begin{align*}
&R_1(x_1,\ldots,x_s) = {\rm id}_{\W(x_1)\otimes \cdots\otimes \W(x_s)},\\
&R_{\sigma_i}(x_1,\ldots,x_s) = \left(\otimes_{j<i}{\rm id}_{\W(x_j)}\right)\otimes R(x_i,x_{i+1}) \otimes \left(\otimes_{j>i+1}{\rm id}_{\W(x_j)}\right),\\
&R_{ww'}(x_1,\ldots,x_s) = R_{w'}(x_{w(1)},\ldots,x_{w(s)})R_{w}(x_1,\ldots,x_s),
\end{align*}
for $w, w'\in \mathfrak{S}_s$ with $\ell(ww')=\ell(w)+\ell(w')$ where $\ell(w)$ denotes the length of $w$ and $R(x,y)=R_{\W,\W}(x,y)$.
To construct a KR module $\W_s$ for $s\ge2$, we need to set 
$x_i=q^{d_r(s-2i+1)}=q_r^{s-2i+1}$, where $d_r$ is determined in the beginning of \S2.1.
In particular, $d_r=1$ for all $r\in I_0$ for untwisted ADE cases.
Hence we have a $\U$-linear map $R_s = R_{\sigma_0}(x_1,\ldots,x_s)$:
\begin{equation*}
R_s : \W(q_r^{s-1})\ot \dots\ot  \W(q_r^{1-s}) \longrightarrow \W(q_r^{1-s})\ot \dots\ot \W(q_r^{s-1}).
\end{equation*}
Here $\sigma_0$ is the longest element in $\mathfrak{S}_s$.
Now we define a KR module corresponding to $\W$ and $s$ by
\begin{equation}\label{W_s}
\W_s = \Im R_s.
\end{equation}

For a KR module $\W_s$ corresponding to $\varpi_r$ ($r\in I_0$),
we also define the dual KR module $\W^\vee_s$ as follows. 
Let $W$ be the affine Weyl group and $W_0$ its subgroup generated by 
$\{s_i\mid i\in I_0\}$ where $s_i$ stands for the simple reflection for $\alpha_i$.
Both act on $P_{\mathrm{cl}}$. Let $w_0$ be the longest element of $W_0$.
For $r\in I_0$, we define $r^\vee\in I_0$ by $-w_0\varpi_r=\varpi_{r^\vee}$.
From this fixed $r^\vee$, we set $\W^\vee=W(\varpi_{r^\vee})$. $\W^\vee_s$ is
constructed similarly from $\W^\vee$ by the fusion construction.

\subsection{Quantum $R$-matrix}
Fusion construction is used not only for defining KR modules
but also for giving a quantum $R$-matrix for two KR modules. 
For two KR modules $\W_s,\W'_{s'}$, define a linear map $R_{\W_s,\W'_{s'}}(x,y)$ from
\begin{equation} \label{ambient}
\W_s(x)\ot \W'_{s'}(y)\subset
(\W(q_r^{1-s}x)\ot \dots\ot \W(q_r^{s-1}x))\ot
(\W'(q_{r'}^{1-s'}y)\ot \dots\ot \W'(q_{r'}^{s'-1}y))
\end{equation}
by
\[
R_{\W_s,\W'_{s'}}(x,y)=(R_{s,s+1}\cdots R_{2,3}R_{1,2})
\cdots
(R_{s+s'-2,s+s'-1}\cdots R_{s,s+1}R_{s-1,s})
(R_{s+s'-1,s+s'}\cdots R_{s+1,s+2}R_{s,s+1}),
\]
where $R_{i,i+1}$ stands for the quantum $R$-matrix acting only on the $i$-th and
$i+1$-th components of the right hand side of \eqref{ambient}. Hence, it maps to 
\[
(\W'(q_{r'}^{1-s'}y)\ot \dots\ot \W'(q_{r'}^{s'-1}y))\ot
(\W(q_r^{1-s}x)\ot \dots\ot \W(q_r^{s-1}x)),
\]
namely, it interchanges $\W(q_r^{1-s}x)\ot \dots\ot \W(q_r^{s-1}x)$ and
$\W'(q_{r'}^{1-s'}y)\ot \dots\ot \W'(q_{r'}^{s'-1}y)$.

\begin{proposition}
\begin{itemize}
\item[(i)] The image of $R_{\W_s,\W'_{s'}}(x,y)$ belongs to $\W'_{s'}(y)\ot\W_s(x)$.
\item[(ii)] They satisfy the Yang-Baxter equation:
\[
R_{\W'_{s'},\W''_{s''}}(y,z)R_{\W_s,\W''_{s''}}(x,z)R_{\W_s,\W'_{s'}}(x,y)
=R_{\W_s,\W'_{s'}}(x,y)R_{\W_s,\W''_{s''}}(x,z)R_{\W'_{s'},\W''_{s''}}(y,z).
\]
\end{itemize}
\end{proposition}
\begin{proof}
To prove (i), we slightly generalize the construction of $R_w(x_1,\ldots,x_s)$ in 
\eqref{R_w}. This time, for two fundamental modules $\W,\W'$ and
$x_1,\ldots,x_s,y_1,\ldots,y_{s'}\in\C(q)$, we consider the $\U$-linear map 
\begin{align*}\label{R'_w}
R'_{\sigma_0^{(s+s')}}(x_1,\ldots,x_s,y_1,\ldots,y_{s'}):&\,
\W(x_1)\ot\cdots\ot\W(x_s)\ot\W'(y_1)\ot\cdots\ot\W'(y_{s'})\\
&\longrightarrow \W'(y_{s'})\ot\cdots\ot\W'(y_1)\ot\W(x_s)\ot \cdots\ot \W(x_1).
\end{align*}
Here $\sigma_0^{(s+s')}$ is the longest element of $\mathfrak{S}_{s+s'}$.
Let $\sigma_0^{(s)}$ (resp. $\sigma_0^{(s')}$) be the longest element of 
the subgroup $\mathfrak{S}_s$ of the former $s$ letters $\{1,\ldots,s\}$ 
(resp. $\mathfrak{S}_{s'}$ of the latter $s'$ letters $\{s+1,\ldots,s+s'\}$)
of $\mathfrak{S}_{s+s'}$. Then $\sigma_0^{(s)}$ and $\sigma_0^{(s')}$ commute 
with each other. Let $\bar{\sigma}_0^{(s)}$ (resp. $\bar{\sigma}_0^{(s')}$) again be
the longest element of $\mathfrak{S}_s$ but of the latter $s$ letters $\{s'+1,\ldots,s+s'\}$
(resp. of the former $s'$ letters $\{1,\ldots,s'\}$). 
$\bar{\sigma}_0^{(s)}$ and $\bar{\sigma}_0^{(s')}$ also commute with each other. Set
\[
\tau=\left(\begin{array}{cccccc}
1&\cdots&s&s+1&\cdots&s+s'\\
s'+1&\cdots&s+s'&1&\cdots&s'
\end{array}\right).
\]
Then we have 
$\tau\sigma_0^{(s')}\sigma_0^{(s)}=\bar{\sigma}_0^{(s')}\bar{\sigma}_0^{(s)}\tau
=\sigma_0^{(s+s')}$. The corresponding relation for the quantum $R$-matrices 
$R_w$'s verifies the assertion. 

(ii) is shown by successive uses of the Yang-Baxter equations among quantum 
$R$-matrices $R_{\W,\W'}$, $R_{\W,\W''},R_{\W',\W''}$ with various parameters.
\end{proof}

\subsection{Quantum $K$-matrix} \label{subsec:quantum K}
Quantum $K$-matrix is a solution to the reflection equation. Let $\W$ be a fundamental
module. Following \cite{RV}, we consider two cases:
Untwisted case
\begin{equation} \label{untwisted}
K_{\W}(x):\W(x)\rightarrow \W(x^{-1}),
\end{equation}
and twisted case
\begin{equation} \label{twisted}
K_{\W}(x):\W(x)\rightarrow \W^\vee(x^{-1}).
\end{equation}
To deal with two cases together, we introduce symbols $*=\emptyset,\vee$, so that
one can write uniformly as $K_{\W}(x):\W(x)\rightarrow\W^*(x^{-1})$. Under this notation,
for two fundamental modules $\W,\W'$, the reflection equation reads as 
\begin{equation} \label{refl eq}
R_{\W'^*,\W^*}(y^{-1},x^{-1})K_{\W'}(y)R_{\W^*,\W'}(x^{-1},y)K_{\W}(x)
=K_{\W}(x)R_{\W'^*,\W}(y^{-1},x)K_{\W'}(y)R_{\W,\W'}(x,y)
\end{equation}
as a map from $\W(x)\ot\W'(y)$ to $\W^*(x^{-1})\ot\W'^*(y^{-1})$.

We define a quantum $K$-matrix 
for the KR module $\W_s$ by composing the ones and the quantum $R$
matrices for fundamental modules by
\begin{align*}
K_{\W_s}(x)=&K^{(s)}(x)\cdots K^{(2)}(x)K^{(1)}(x),\\
K^{(j)}(x)=&K_1(q_r^{-s+2j-1}x)R^*_{1,2}(q_r^{s-2j+3}x^{-1},q_r^{-s+2j-1}x)
R^*_{2,3}(q_r^{s-2j+5}x^{-1},q_r^{-s+2j-1}x)\cdots \\
&\hspace{80mm}\cdots R^*_{j-1,j}(q_r^{s-1}x^{-1},q_r^{-s+2j-1}x).
\end{align*}
Here $R^*_{i,i+1}(x,y)$ is the quantum $R$-matrix $R^*_{\W,\W}(x,y)$ acting on the
$i$-th and $i+1$-th components.
$K_{\W_s}(x)$ is a map from $\W(q_r^{1-s}x)\ot\cdots\ot\W(q_r^{s-1}x)$
to $\W^*(q_r^{1-s}x^{-1})\ot\cdots\ot\W^*(q_r^{s-1}x^{-1})$.
We can also write inductively as
\begin{align}
K_{\W_s}(x)&=K_1(q_r^{s-1}x)R^*_{1,2}(q_r^{-s+3}x^{-1},q_r^{s-1}x)
R^*_{2,3}(q_r^{-s+5}x^{-1},q_r^{s-1}x)\cdots \nonumber \\
&\hspace{55mm}\cdots R^*_{s-1,s}(q_r^{s-1}x^{-1},q_r^{s-1}x)
(K_{\W_{s-1}}(q_r^{-1}x)\ot1).
\label{rec for K}
\end{align}
\begin{figure}[h]
\begin{picture}(300,185)(-32,5)
\put(50,10){
\line(0,1){175}
\put(0,30){\put(0,20){\line(2,1){100}}\put(0,20){\vector(2,-1){100}}}
\put(0,65){\put(0,20){\line(2,1){100}}\put(0,20){\vector(2,-1){100}}}
\put(0,100){\put(0,20){\line(2,1){100}}\put(0,20){\vector(2,-1){100}}}
\put(-42,118){$K(q_r^{-2}x)$}
\put(-35,82){$K(x)$}
\put(-38,46){$K(q_r^2x)$}
\put(18,112){$R^*(q_r^2x^{-1},x)$}
\put(44,91){$R^*(q_r^2x^{-1},q_r^2x)$}
\put(18,56){$R^*(x^{-1},q_r^2x)$}
\put(110,168){$\W(q_r^{-2}x)$}
\put(110,133){$\W(x)$}
\put(110,98){$\W(q_r^2x)$}
\put(110,67){$\W^*(q_r^2x^{-1})$}
\put(110,32){$\W^*(x^{-1})$}
\put(110,-3){$\W^*(q_r^{-2}x^{-1})$}
}
\end{picture}
\caption{Graphical representation for $K_{\W_3}(x)$}
\end{figure}
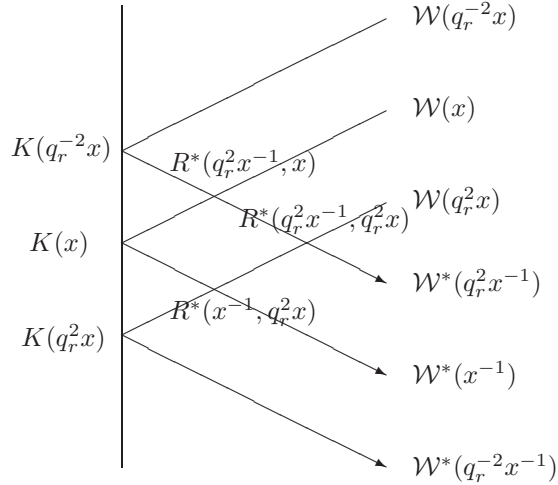
\begin{proposition} \label{prop: quantum K satisfy refl eq}
\begin{itemize}
\item[(i)] The image of $K_{\W_s}(x)$ belongs to $\W^*_s(x^{-1})$.
\item[(ii)] The quantum $K$-matrices satisfy the reflection equation:
\[
R_{\W'^*_{s'},\W^*_{s}}(y^{-1},x^{-1})K_{\W'_{s'}}(y)R_{\W_s^*,\W'_{s'}}(x^{-1},y)K_{\W_s}(x)
=K_{\W_{s}}(x)R_{\W'^*_{s'},\W_{s}}(y^{-1},x)K_{\W'_{s'}}(y)R_{\W_s,\W'_{s'}}(x,y).
\]
\end{itemize}
\end{proposition}
\begin{proof}
To prove (i), we note that $K_{\W_s}(x)R_s=R_s\bar{K}_{\W_s}(x)$ holds, where 
$\bar{K}_{\W_s}(x)$ is defined by replacing $q_r$ with $q_r^{-1}$ in $K_{\W_s}(x)$.
It can be proven by successive uses of the Yang-Baxter equation and the reflection 
equation \eqref{refl eq}. (ii) is also shown by successive uses of \eqref{refl eq}.
\end{proof}

\section{Crystals and $\imath$crystals}
\subsection{Crystals}
A crystal is a set $\B$ equipped with maps $\tE_i,\tF_i:\B\rightarrow\B\sqcup\{0\}$
for $i\in I$, where $0$ is a formal symbol. By depicting $b\stackrel{i}{\rightarrow}b'$
when $\tF_ib=b'$ for $b,b'\in\B$, $\B$ becomes a colored oriented graph called 
crystal graph. We set $\vep_i(b)=\max\{m\ge0\mid \tE_i^mb\ne0\},
\vphi_i(b)=\max\{m\ge0\mid \tF_i^mb\ne0\}$ for $b\in\B$. 

Let $\B$ and $\B'$ be crystals. Then $\B\ot\B'$ also has the structure of crystal by
\begin{align*}
&\vep_i(b_1\ot b_2)=\vep_i(b_2)+(\vep_i(b_1)-\vphi_i(b_2))_+,\\
&\vphi_i(b_1\ot b_2)=\vphi_i(b_1)+(\vphi_i(b_2)-\vep_i(b_1))_+,\\
&\tE_i(b_1\ot b_2)=\left\{
\begin{array}{ll}
\tE_ib_1\ot b_2\quad&\text{if }\vep_i(b_1)>\vphi_i(b_2),\\
b_1\ot\tE_ib_2\quad&\text{if }\vep_i(b_1)\le\vphi_i(b_2),
\end{array}\right.\\
&\tF_i(b_1\ot b_2)=\left\{
\begin{array}{ll}
\tF_ib_1\ot b_2\quad&\text{if }\vep_i(b_1)\ge\vphi_i(b_2),\\
b_1\ot\tF_ib_2\quad&\text{if }\vep_i(b_1)<\vphi_i(b_2).
\end{array}\right.
\end{align*}
It is called the tensor product of crystals, though it is the Cartesian product of two sets.
Note that this convention is opposite to \cite{Ka:1991}.

The notion of crystal was abstracted from that of crystal basis. We again use a different
convention from \cite{Ka:1991}, namely, we replace $q$ with $q^{-1}$. Set
$\mathbf{A}=\{f(q)\in\C(q)\mid f(q)\text{ has no pole at }q=\infty\}$ and 
let $M$ be a $\U$-module. Then a crystal basis of $M$ is a pair $(\mathcal{L},\B)$ of 
a $\mathbf{A}$-lattice $\mathcal{L}$ 
and a $\C$-basis $\B$ of $\mathcal{L}/q^{-1}\mathcal{L}$. Moreover, operators
$\tE_i,\tF_i$ have concrete meanings on $M$ and we have $\tE_i\mathcal{L}\subset
\mathcal{L},\tF_i\mathcal{L}\subset\mathcal{L}$, so $\tE_i\B\subset\B\sqcup\{0\},
\tF_i\B\subset\B\sqcup\{0\}$ where $0\in M$.

It is known that a KR module $\W_s$ introduced in section \ref{subsec:KRmodule} has a
crystal basis when our Cartan matrix $A$ is of nonexceptional affine type \cite{OS}.
We call it a KR crystal and denote it by $\B_s$.
We remark that because of the convention of crystal basis with $q$ replaced by $q^{-1}$,  
our tensor product rule of crystals is consistent with our coproduct \eqref{coprod}.
Let $\B_s$ a KR crystal. One can consider its affinization $\B_s(x)=\{b\ot x^d\mid b\in
\B_s,d\in\Z\}$. We often write $x^d b$ instead of $b \otimes x^d$. Crystal operators
$\tilde{E}_i,\tilde{F}_i$ act on $\B_s(x)$ as $\tilde{E}_i(x^db)=x^{d+\delta_{i0}}\tilde{E}_ib,
\tilde{F}_i(x^db)=x^{d-\delta_{i0}}\tilde{F}_ib$.
Let $\B_s,\B'_{s'}$ be KR crystals. Then, there exists an isomorphism of crystals
$\mathcal{R}_{\B_s,\B'_{s'}}:\B_s(x)\ot\B'_{s'}(y)\rightarrow\B'_{s'}(y)\ot\B_s(x)$, namely,
a map that commutes with the action of crystal operators. By taking a suitable limit
of $q$, $q\to\infty$ in our case, we obtain a set-theoretical Yang-Baxter equation
\[
\mathcal{R}_{\B'_{s'},\B''_{s''}}(y,z)\mathcal{R}_{\B_s,\B''_{s''}}(x,z)\mathcal{R}_{\B_s,\B'_{s'}}(x,y)
=\mathcal{R}_{\B_s,\B'_{s'}}(x,y)\mathcal{R}_{\B_s,\B''_{s''}}(x,z)\mathcal{R}_{\B'_{s'},\B''_{s''}}(y,z).
\]
from \eqref{YBeq}.

For a crystal $\B$ there is a notion of its dual crystal $\B^\vee$~\cite{Ka:1994}.
It is defined by $\B^\vee=\{b^\vee\mid b\in \B\}$ with
\begin{equation} \label{dual crystal}
\tilde{E}_ib^\vee= (\tilde{F}_ib)^\vee, \qquad
\tilde{F}_ib^\vee = (\tilde{E}_ib)^\vee.
\end{equation}
Although it is not written in the literature except in simpler cases, the dual crystal
of the KR crystal $\B_s$ is given by the KR crystal of the KR module $\W_s^\vee$
introduced in section \ref{subsec:KRmodule}. We will see examples in 
type A in section \ref{sec:type A}.

\subsection{$\imath$Crystals} \label{subsec:icrystals}
In this subsection, we briefly recall the notion of $\imath$crystals, which is an analogous notion to crystal for the $\imath$quantum group. It was introduced in \cite{Wa2} under the following assumptions:
\begin{itemize}
  \item[(A1)] $I_\bullet = \emptyset$.
  \item[(A2)] $a_{i,\tau(i)} \in \{2,0,-1\}$ for all $i \in I$.
  \item[(A3)] If $a_{i,\tau(i)} = 2$, then $\varsigma_i = q_i^{-1}$ and $\kappa_i = \frac{q_i^{s_i}-q_i^{-s_i}}{q_i-q_i^{-1}}$ for some $s_i \in \mathbb{Z}$.
  \item[(A4)] If $a_{i,\tau(i)} = 0$, then $\varsigma_i = 1$ and $\kappa_i = 0$.
  \item[(A5)] If $a_{i,\tau(i)} = -1$, then $\varsigma_i = q_i^{s_i}$, $\varsigma_i \varsigma_{\tau(i)} = q_i$, and $\kappa_i = 0$ for some $s_i \in \mathbb{Z}$.
\end{itemize}
Therefore, we always keep this assumption whenever we consider $\imath$crystals.

An $\imath$crystal is a set $\B$ equipped with structure maps $\mathrm{wt}^{\imath}$, $\beta_i$, and $\tilde{B}_i$ for $i \in I$. As a special case of \cite[Corollary 5.2.2]{Wa2}, we obtain the following.

\begin{proposition}
  Let $\B$ be a {\rm KR} crystal. Then, it has an $\imath$crystal structure as follows: Let $b \in \B$ and $i \in I$.
  \begin{itemize}
    \item $\mathrm{wt}^{\imath}(b) = \overline{\operatorname{wt}(b)}$, where $\overline{\lambda}$ denote the image of $\lambda \in P_{\text{cl}}$ in $P_{\text{cl}}^\imath := P_{\text{cl}}/\{\lambda + \tau(\lambda) \mid \lambda \in P_{\text{cl}} \}$.
    \item If $a_{i,\tau(i)} = 2$, then
    \begin{align*}
      &\beta_i(b) = \begin{cases}
        \varepsilon_i(b) + 1 & \text{ if } |s_i| \leq \varphi_i(b) \text{ and } s_i - \varphi_i(b) \text{ is odd}, \\
        |s_i| - \langle h_i,\operatorname{wt}(b) \rangle & \text{ if } |s_i| > \varphi_i(b), \\
        \varepsilon_i(b) & \text{ if } |s_i| \leq \varphi_i(b) \text{ and } s_i - \varphi_i(b) \text{ is even},
      \end{cases} \\
      &\tilde{B}_i b = \begin{cases}
        \tilde{F}_i b & \text{ if } |s_i| \leq \varphi_i(b) \text{ and } s_i - \varphi_i(b) \text{ is odd}, \\
        \mathrm{sgn}(s_i)b & \text{ if } |s_i| > \varphi_i(b), \\
        \tilde{E}_i b & \text{ if } |s_i| \leq \varphi_i(b) \text{ and } s_i - \varphi_i(b) \text{ is even},
      \end{cases} \\
  \end{align*}
  where $\mathrm{sgn}(n)$ denotes the sign of $n \in \mathbb{Z}\setminus\{0\}$.
  \item If $a_{i,\tau(i)} = 0$, then
  \begin{align*}
    &\beta_i(b) = \max(\varphi_i(b), \varphi_{\tau(i)}(b)) - \langle h_{\tau(i)}, \operatorname{wt}(b) \rangle, \\
    &\tilde{B}_ib = \begin{cases}
      \tilde{F}_i b & \text{ if } \varphi_i(b) > \varphi_{\tau(i)}(b), \\
      \tilde{E}_{\tau(i)} b & \text{ if } \varphi_i(b) \leq \varphi_{\tau(i)}(b).
    \end{cases}
  \end{align*}
  \item If $a_{i,\tau(i)} = -1$, then
  $$
  \beta_i(b) = \max(\varphi_i(b), \varphi_{\tau(i)}(b)+s_i) - s_i - \langle h_{\tau(i)}, \operatorname{wt}(b) \rangle.
  $$
  \begin{itemize}
    \item When $\varphi_i(b) > \varphi_{\tau(i)}(b) + s_i$,
    $$
    \tilde{B}_i b = \begin{cases}
      \frac{1}{\sqrt{2}} \tilde{F}_i b & \text{ if } \varphi_i(b) = \varphi_{\tau(i)}(b) + s_i + 1 \text{ and } \varphi_{\tau(i)}(\tilde{F}_i b) = \varphi_{\tau(i)}(b)+1, \\
      \tilde{F}_i b & \text{ otherwise}.
    \end{cases}
    $$
    \item When $\varphi_i(b) \leq \varphi_{\tau(i)}(b) + s_i$,
    $$
    \tilde{B}_i b = \begin{cases}
      \frac{1}{\sqrt{2}} \tilde{E}_{\tau(i)} b & \text{ if } \varphi_i(b) = \varphi_{\tau(i)}(b)+s_i \text{ and } \varphi_i(\tilde{E}_{\tau(i)} b) = \varphi_i(b), \\
      \frac{1}{\sqrt{2}}(\tilde{E}_{\tau(i)}b + \tilde{F}_i b) & \text{ if } \varphi_i(b) = \varphi_{\tau(i)}(b)+s_i > (-s_{\tau(i)})_+ \text{ and } \varphi_i(\tilde{E}_{\tau(i)}b) = \varphi_i(b)-1, \\
      \tilde{E}_{\tau(i)} b & \text{ otherwise}.
    \end{cases}
    $$
  \end{itemize}
  \end{itemize}
\end{proposition}

\section{Existence of the combinatorial $K$-matrix}

In this section, we define the notion of combinatorial $K$-matrix in a similar way to that of combinatorial $R$-matrix.
Then, we show that each combinatorial $K$-matrix satisfies the set-theoretical reflection equation.
After that, we give an interpretation of the combinatorial $K$-matrix from a viewpoint of $\imath$crystals.

\subsection{Definition}
Let $\W_s$ be a KR module, $x$ a formal spectral parameter, and $K_{\W_s}(x): \W_s(x) \rightarrow \W_s^*(x^{-1})$ the quantum $K$-matrix.
Recall that $*=\emptyset,\vee$ depending on the Satake diagram we consider.
Let $(\mathcal{L}_s, \B_s)$ denote the crystal basis of $\W_s$.
Set $\mathcal{L}_s(x) := \mathcal{L}_s \otimes_{\mathbb{C}} \mathbb{C}(x)$ and $\B_s(x) := \{ b \otimes x^d \mid b \in \B_s,\ d \in \mathbb{Z} \}$.
We often write $x^d b$ instead of $b \otimes x^d$.
As seen in section \ref{subsec:quantum K}, matrix coefficients of our quantum $K$-matrix $K_{\W_s}(x)$ are rational functions in $q,x$. Hence, by multiplying a suitable power of $q$, one can normalize it in a way such that $K_{\W_s}(x)(\mathcal{L}_s(x)) \subset \mathcal{L}^*_s(x^{-1})$ and the induced $\mathbb{C}(x)$-linear map $\bar{K}_{\W_s}(x): \mathcal{L}_s(x)/q^{-1} \mathcal{L}_s(x) \rightarrow \mathcal{L}^*_s(x^{-1})/q^{-1} \mathcal{L}^*_s(x^{-1})$ is not zero.
Set $\K_{\B_s}(x) = \bar{K}_{\W_s}(x)|_{\B_s(x)}$. If $\K_{\B_s}(x)$ gives rise to a bijection $\B_s(x) \rightarrow \B^*_x(x^{-1})$, then we call it a \emph{combinatorial $K$-matrix} on $\B_s$.

\begin{lemma}
If a combinatorial $K$-matrix exists, then it is unique up to multiplication of $x^d$ for some $d \in \mathbb{Z}$.
\end{lemma}

\begin{proof}
Let $K_1,K_2$ be normalizations of $K_{\W_s}$ which give rise to combinatorial $K$-matrices $\K_1, \K_2$.
Then, there exists a unique $f \in \mathbb{C}[\![q^{-1}]\!](x) \cap \mathbb{C}(q,x)$ such that $K_2 = f K_1$.
Hence, we have
\begin{align}
\K_2 = f_0 \K_1, \label{eq: K2 = f K1}
\end{align}
where $f_0 \in \mathbb{C}(x)$ denotes the constant term of $f$.
In order to prove the assertion, we need to show that $f_0 = x^d$ for some $d \in \mathbb{Z}$.

Let $b \in \B_s$.
Then, we have $\K_1(b), \K_2(b) \in \B_s^*(x^{-1})$.
On the other hand, by equation \eqref{eq: K2 = f K1}, we have
$$
\K_2(b) = f_0 \K_1(b).
$$
Since $\B_s^*$ is a basis of the $\mathbb{C}$-vector space $\mathcal{L}(\W_s^*)/q^{-1}\mathcal{L}(\W_s^*)$, the fact that both $\K_1(b)$ and $f_0 \K_1(b)$ belong to $\B_s^*(x^{-1})$ implies that $f_0 = x^d$ for some $d \in \mathbb{Z}$.
Thus, the proof completes.
\end{proof}

\subsection{Set-theoretical reflection equation}
Let $\W_s, \W'_{s'}$ be KR modules, and $\B_s, \B'_{s'}$ their crystal bases.
Assume that the combinatorial $K$-matrices $\K_{\B_s}(x), \K_{\B'_{s'}}(y)$ exist.

\begin{proposition}\label{prop: comb K satisfy set reflection eq}
The combinatorial $K$-matrices $\K_{\B_s}(x)$ and $\K_{\B'_{s'}}(y)$ satisfy the set-theoretical reflection equation
\begin{equation} \label{comb RE}
\mathcal{R}_{\B^{\prime *}_{s'}, \B^*_s}(y^{-1}, x^{-1}) \K_{\B'_{s'}}(y) \mathcal{R}_{\B^*_s, \B'_{s'}}(x^{-1}, y) \K_{\B_s}(x) = \K_{\B_s}(x) \mathcal{R}_{\B^{\prime *}_{s'}, \B_s}(y^{-1}, x) \K_{\B'_{s'}}(y) \mathcal{R}_{\B_s, \B'_{s'}}(x,y).
\end{equation}
\end{proposition}

\begin{proof}
The assertion is clear from Proposition \ref{prop: quantum K satisfy refl eq}.
\end{proof}

\subsection{$\imath$Crystal theoretical viewpoint}
Let $\W_s$ be a KR module and $\B_s$ its crystal basis.
Assume that the combinatorial $K$-matrix $\K_{\B_s}(x)$ exists.

\begin{proposition}\label{prop: comb K is icry isom}
The combinatorial $K$-matrix is an isomorphism of $\imath$crystals.
\end{proposition}

\begin{proof}
Immediate from the definitions.
\end{proof}

By Propositions \ref{prop: comb K satisfy set reflection eq} and \ref{prop: comb K is icry isom}, we see that when finding the combinatorial $K$-matrix, it would be meaningful to find an $\imath$crystal isomorphism from $\B_s(x)$ to $\B_s^*(x^{-1})$.
Indeed, as we will see below, in some cases, we can obtain the combinatorial $K$-matrix explicitly in this way without knowing the quantum $K$-matrix.
This can be seen as a natural generalization of the fact that the combinatorial $R$-matrix can be constructed from the tensor product of two KR crystals, without knowing the quantum $R$-matrix.
This fact suggests that there is a general theory which ensures the existence of combinatorial $K$-matrix by means of $\imath$crystals.

\section{Type A case} \label{sec:type A}
  In this section, we consider the quantum affine algebra of type $A^{(1)}_{n-1}$ with 
$n \geq 3$, the simplest family of KR modules $\W_s$ and its dual $\W_s^\vee$, and investigate
quantum/combinatorial $K$-matrices for all quasi-split cases.

\subsection{The simplest KR modules and their $R$-matrices}
  The fundamental module $\W$ associated to the first level $0$ fundamental weight
$\varpi_1$ is the vector representation $\W = \bigoplus_{i=1}^n \mathbb{C}(q)v_i$.
Since $-w_0\varpi_1=\varpi_{n-1}$, its dual representation $\W^\vee = \bigoplus_{i=1}^n \mathbb{C}(q) v^\vee_i$ has a highest weight $\varpi_{n-1}$.
The actions of Chevalley generators on $\W$ and $\W^\vee$ are given as follows.
\begin{align*}
&E_iv_j=\delta_{i+1,j}v_i,\quad F_iv_j=\delta_{ij}v_{i+1},\quad K_iv_j=q^{\delta_{ij}-\delta_{i+1,j}}v_j,\\
&E_iv_j^\vee=\delta_{ij}v_{i+1}^\vee,\quad F_iv_j^\vee=\delta_{i+1,j}v_i^\vee,\quad K_iv_j^\vee=q^{-\delta_{ij}+\delta_{i+1,j}}v_j^\vee.
\end{align*}
In this section, indices $i,j$ like in the above should be considered in $\Z/n\Z$.
The crystal basis of $\W$ is $\B = \{ b_1,b_2,\ldots,b_n \}$ and its crystal 
graph is given as follows:
  $$
  \xymatrix{
    b_1 \ar[r]^{1} & b_2 \ar[r]^{2} & \cdots \ar[r]^{n-1} & b_n \ar@(dr,dl)[lll]^{0}
  }
  $$
The crystal basis of $\W^\vee$ is $\B^\vee := \{ b^\vee_1,b^\vee_2,\ldots,b^\vee_n \}$ and its crystal graph is obtained from that of $\B$ by reversing the arrows:
  $$
  \xymatrix{
    b^\vee_1 \ar@(dl,dr)[rrr]_{0} & b^\vee_2 \ar[l]_{1} & \cdots \ar[l]_{2} & b^\vee_n \ar[l]_{n-1}
  }
  $$

The quantum $R$-matrices among $\W$ and $\W^\vee$ are given as follows \cite{DO}.
\begin{align*}
R_{\W,\W}(x,y)&=\sum_iE_{ii}\ot E_{ii}+\sum_{i\ne j}\frac{1-q^2}{1-q^2z}z^{\chi(i<j)}
E_{ii}\ot E_{jj}+\sum_{i\ne j}q\frac{1-z}{1-q^2z}E_{ij}\ot E_{ji},\\
R_{\W^\vee,\W}(x,y)&=\sum_{i\ne j}(-q)^{i-j-1}((-q)^nz)^{\chi(i<j)}E_{ij}\ot E_{ij}
+\sum_{i,j}\frac{q^{\delta_{ij}}-q^{-\delta_{ij}}(-q)^nz}{1-q^2}E_{ij}\ot E_{ji},\\
R_{\W^\vee,\W^\vee}(x,y)&=\sum_iE_{ii}\ot E_{ii}+\sum_{i\ne j}\frac{1-q^2}{1-q^2z}z^{\chi(i>j)}
E_{ii}\ot E_{jj}+\sum_{i\ne j}q\frac{1-z}{1-q^2z}E_{ij}\ot E_{ji}.
\end{align*}
Here $z=x/y$ and $E_{ij}$ is the linear operator such that $E_{ij}v_j^*=v_i^{*'}$ when it belongs to 
$\mathrm{Hom}(\W^*,\W^{*'})$ ($*,*'=\emptyset,\vee$).

Next we describe $\W_s$ obtained from $\W$ by fusion construction. The linear space
$\W^{\ot s}$ has a standard basis $\{v_{\mathbf{i}}\mid \mathbf{i}\in\{1,\ldots,n\}^s\}$. Here,
for $\mathbf{i}=(i_1,\ldots,i_s)$, we set $v_{\mathbf{i}}=v_{i_1}\ot\cd\ot v_{i_s}$. Define the weight of $\mathbf{i}$
$\mathrm{wt}(\mathbf{i})$ by
$\alpha=(\alpha_1,\ldots,\alpha_n)$ where $\alpha_i$ is the number of $i$'s in $\mathbf{i}$.
For $\alpha\in\{(\alpha_1,\ldots,\alpha_n)\mid \alpha_i\ge0,\sum_{i=1}^n\alpha_i=s\}$, define
\[
v_\alpha=\sum_{\mathrm{wt}(v_{\mathbf{i}})=\alpha} q^{-\tau(\mathbf{i})}v_{\mathbf{i}},
\qquad \tau(\mathbf{i})=\sharp\{(p,q)\mid1\le p<q\le n,i_p>i_q\}.
\]
Then $\{v_\alpha\mid \alpha_i\ge0,\sum_{i=1}^n\alpha_i=s\}$ turns out a basis
of the KR module $\W_s$ defined in \eqref{W_s}. The actions of Chevalley generators are
given by
\[
E_iv_\alpha=[\alpha_i+1]v_{\alpha+\eb_i-\eb_{i+1}},\quad 
F_iv_\alpha=[\alpha_{i+1}+1]v_{\alpha-\eb_i+\eb_{i+1}},\quad
K_iv_\alpha=q^{\alpha_i-\alpha_{i+1}}v_{\alpha}.
\]
Here $\eb_i$ is the standard basis vector and subscripts of $\alpha$ should be considered
modulo $n$. If $\alpha$ contains a negative integer upon application, then $v_\alpha$ 
should be considered as 0.

Similarly, one can realize a basis of $\W_s^\vee$ as a linear combination of $v_{\mathbf{i}}^\vee
=v_{i_1}^\vee\ot\cdots\ot v_{i_s}^\vee\in\W^{\vee\ot s}$. Set
\[
v_\alpha^\vee=\sum_{\mathrm{wt}(\mathbf{i})=\alpha} q^{-\bar{\tau}(\mathbf{i})}v_{\mathbf{i}}^\vee,
\qquad \bar{\tau}(\mathbf{i})=\sharp\{(p,q)\mid1\le p<q\le n,i_p<i_q\}.
\]
$\{v_\alpha^\vee\mid \alpha_i\ge0,\sum_{i=1}^n\alpha_i=s\}$ turns out a basis
of $\W_s^\vee$. The actions of Chevalley generators are
given by
\[
E_iv_\alpha^\vee=[\alpha_{i+1}+1]v_{\alpha-\eb_i+\eb_{i+1}}^\vee,\quad 
F_iv_\alpha^\vee=[\alpha_i+1]v_{\alpha+\eb_i-\eb_{i+1}}^\vee,\quad
K_iv_\alpha^\vee=q^{-\alpha_i+\alpha_{i+1}}v_{\alpha}^\vee.
\]

We also review the piece-wise linear formulas for the combinatorial $R$-matrices.
See e.g. \cite{KOY}, but note that the tensor product rule in this paper is opposite to it.
\begin{align*}
&\mathcal{R}_{\B_s,\B_{s'}}(x,y)(x^db_\alpha\ot y^eb_\beta)
=y^{e+Q_0(\alpha,\beta)}b_{\beta'}\ot x^{d-Q_0(\alpha,\beta)}b_{\alpha'} \\
&\qquad\beta'_i=\beta_i+Q_i(\alpha,\beta)-Q_{i-1}(\alpha,\beta),\quad
\alpha'_i=\alpha_i+Q_{i-1}(\alpha,\beta)-Q_i(\alpha,\beta), \\
&\mathcal{R}_{\B_s^\vee,\B_{s'}}(x,y)(x^db_\alpha^\vee\ot y^eb_\beta)
=y^{e+Q_0(\alpha,\beta)}b_{\beta'}\ot x^{d-Q_0(\alpha,\beta)}b_{\alpha'}^\vee \\
&\qquad\beta'_i=\beta_i+P_i(\alpha,\beta)-P_{i-1}(\alpha,\beta),\quad
\alpha'_i=\alpha_i+P_i(\alpha,\beta)-P_{i-1}(\alpha,\beta), \\
&\mathcal{R}_{\B_s^\vee,\B_{s'}^\vee}(x,y)(x^db_\alpha^\vee\ot y^eb_\beta^\vee)
=y^{e+Q_0(\beta.\alpha)}b_{\beta'}^\vee\ot x^{d-Q_0(\beta,\alpha)}b_{\alpha'}^\vee \\
&\qquad\beta'_i=\beta_i+Q_{i-1}(\beta,\alpha)-Q_i(\beta,\alpha),\quad
\alpha'_i=\alpha_i+Q_i(\alpha,\beta)-Q_{i-1}(\alpha,\beta),
\end{align*}
where $Q_i,P_i$ are given by
\[
Q_i(\alpha,\beta)=\min_{1\le k\le n}
(\sum_{j=k+1}^n\alpha_{i+j}+\sum_{k=1}^{k-1}\beta_{i+j}),\qquad
P_i(\alpha,\beta)=\min(\alpha_{i+1},\beta_{i+1}).
\]

Let $\mathcal{L}_s^*$ ($* \in \{ \emptyset, \vee \}$) denote the crystal lattice of $\W_s$ spanned by $v_{\alpha}$'s.
Set $b^*_\alpha := v_\alpha + q^{-1}\mathcal{L}^*_s$.
Then, $\B^*_s := \{ b^*_\alpha \}$ forms the crystal basis.
We write $\alpha$ instead of $b_\alpha$ if no confusion can occur.
We summarize the crystal and $\imath$crystal structure below. 
Note that for the parameters $s_i$ in (A3) in section \ref{subsec:icrystals}, we set $s_i=0$.
The crystal structure of $\B_s(x) = \{ x^d b_\alpha \mid d \in \mathbb{Z},\ b_\alpha \in \B_s \}$ is as follows:
\begin{align*}
  &\langle h_i, \operatorname{wt}(x^d b_\alpha) \rangle = \alpha_i-\alpha_{i+1}, \\
  &\varepsilon_i(x^d b_\alpha) = \alpha_{i+1}, \\
  &\varphi_i(x^d b_\alpha) = \alpha_i, \\
  &\tilde{E}_i (x^d b_\alpha) = x^{d+\delta_{i,0}} b_{\alpha+\mathbf{e}_i-\mathbf{e}_{i+1}}, \\
  &\tilde{F}_i (x^d b_\alpha) = x^{d-\delta_{i,0}} b_{\alpha-\mathbf{e}_i+\mathbf{e}_{i+1}}.
\end{align*}
Then, the $\imath$crystal structure of $\B_s(x)$ is described as follows (below, $a_{i, \tau(i)} = -1$ occurs only when $n = 2n'+1$ for some $n' \in \mathbb{Z}_{> 0}$ and $i = n',n'+1$):
\begin{itemize}
  \item When $a_{i,\tau(i)} = 2$.
  \begin{align*}
    &\beta_i(x^d b_\alpha) = \alpha_{i+1} + \theta(\alpha_i), \\
    &\tilde{B}_i (x^d b_\alpha) = x^{d+(-1)^{\theta(\alpha_i)}} b_{\alpha + (-1)^{\theta(\alpha_i)}(\mathbf{e}_i - \mathbf{e}_{i+1})}.
  \end{align*}
  \item When $a_{i,\tau(i)} = 0$.
  \begin{align*}
    &\beta_i(x^d b_\alpha) = \begin{cases}
      \alpha_i-\alpha_{\tau(i)}+\alpha_{\tau(i)+1} & \text{ if } \alpha_i > \alpha_{\tau(i)}, \\
      \alpha_{\tau(i)+1} & \text{ if } \alpha_i \leq \alpha_{\tau(i)},
    \end{cases} \\
    &\tilde{B}_i (x^d b_\alpha) = \begin{cases}
      x^{d-\delta_{i,0}} b_{\alpha-\mathbf{e}_i+\mathbf{e}_{i+1}} & \text{ if } \alpha_i > \alpha_{\tau(i)}, \\
      x^{d+\delta_{\tau(i),0}} b_{\alpha+\mathbf{e}_{\tau(i)}-\mathbf{e}_{\tau(i)+1}} & \text{ if } \alpha_i \leq \alpha_{\tau(i)}.
    \end{cases}
  \end{align*}
  \item When $a_{i,\tau(i)} = -1$.
  \begin{align*}
    &\beta_i(x^d b_\alpha) = \begin{cases}
      \alpha_i - \alpha_{\tau(i)} + \alpha_{\tau(i)+1} - s_i, & \text{ if } \alpha_i > \alpha_{\tau(i)}+s_i \\
      \alpha_{\tau(i)+1} & \text{ if } \alpha_i \leq \alpha_{\tau(i)}+s_i,
    \end{cases}, \\
    &\tilde{B}_{n'}(x^d b_\alpha) = \begin{cases}
      \frac{1}{\sqrt{2}} x^d b_{\alpha-\mathbf{e}_{n'}+\mathbf{e}_{n'+1}} & \text{ if } \alpha_{n'} = \alpha_{n'+1} + s_{n'}+1, \\
      x^d b_{\alpha-\mathbf{e}_{n'}+\mathbf{e}_{n'+1}} & \text{ if } \alpha_{n'} > \alpha_{n'+1} + s_{n'}+1, \\
      \frac{1}{\sqrt{2}} x^d b_{\alpha+\mathbf{e}_{n'+1}-\mathbf{e}_{n'+2}} & \text{ if } \alpha_{n'} = \alpha_{n'+1} + s_{n'}, \\
      x^d b_{\alpha+\mathbf{e}_{n'+1}-\mathbf{e}_{n'+2}} & \text{ if } \alpha_{n'} < \alpha_{n'+1} + s_{n'}
    \end{cases}, \\
    &\tilde{B}_{n'+1}(x^d b_\alpha) = \begin{cases}
      x^d b_{\alpha-\mathbf{e}_{n'+1}+\mathbf{e}_{n'+2}} & \text{ if } \alpha_{n'+1} > \alpha_{n'} + s_{n'+1}, \\
      \frac{1}{\sqrt{2}} x^d( b_{\alpha+\mathbf{e}_{n'}-\mathbf{e}_{n'+1}} + b_{\alpha-\mathbf{e}_{n'+1}+\mathbf{e}_{n'+2}}) & \text{ if } \alpha_{n'+1} = \alpha_{n'} + s_{n'+1} > (-s_{n'})_+, \\
      x^d b_{\alpha+\mathbf{e}_{n'}-\mathbf{e}_{n'+1}} & \text{ otherwise}.
    \end{cases}
  \end{align*}
\end{itemize}
Also, the crystal and $\imath$crystal structure of $\B_s^\vee(x^{-1}) = \{ x^d b_\alpha^\vee \mid b_\alpha \in \B_s \}$ is as follows (below, we will not consider $\B^\vee_s(x^{-1})$ when we have $a_{i,\tau(i)} = -1$ for some $i$):
\begin{align*}
  &\langle h_i, \operatorname{wt}(x^d b_\alpha^\vee) \rangle = -\alpha_i+\alpha_{i+1}, \\
  &\varepsilon_i(x^d b_\alpha^\vee) = \alpha_i, \\
  &\varphi_i(x^d b_\alpha^\vee) = \alpha_{i+1}, \\
  &\tilde{E}_i (x^d b_\alpha^\vee) = x^{d-\delta_{i,0}}b^\vee_{\alpha-\mathbf{e}_i+\mathbf{e}_{i+1}}, \\
  &\tilde{F}_i (x^d b_\alpha^\vee) = x^{d+\delta_{i,0}} b^\vee_{\alpha+\mathbf{e}_i-\mathbf{e}_{i+1}}.
\end{align*}
\begin{itemize}
  \item When $a_{i,\tau(i)} = 2$.
  \begin{align*}
    &\beta_i(x^d b_\alpha^\vee) = \alpha_i + \theta(\alpha_{i+1}), \\
    &\tilde{B}_i (x^d b_\alpha^\vee) =  x^{d-(-1)^{\theta(\alpha_{i+1})}} b^\vee_{\alpha - (-1)^{\theta(\alpha_{i+1})}(\mathbf{e}_i - \mathbf{e}_{i+1})}.
  \end{align*}
  \item When $a_{i,\tau(i)} = 0$.
  \begin{align*}
    &\beta_i(x^d b_\alpha^\vee) = \begin{cases}
      \alpha_{i+1}+\alpha_{\tau(i)}-\alpha_{\tau(i)+1} & \text{ if } \alpha_{i+1} > \alpha_{\tau(i)+1}, \\
      \alpha_{\tau(i)} & \text{ if } \alpha_{i+1} \leq \alpha_{\tau(i)+1},
    \end{cases} \\
    &\tilde{B}_i (x^d b_\alpha^\vee) = \begin{cases}
      x^{d+\delta_{i,0}} b^\vee_{\alpha+\mathbf{e}_i-\mathbf{e}_{i+1}} & \text{ if } \alpha_{i+1} > \alpha_{\tau(i)+1}, \\
      x^{d-\delta_{\tau(i),0}} b^\vee_{\alpha-\mathbf{e}_{\tau(i)}+\mathbf{e}_{\tau(i)+1}} & \text{ if } \alpha_{i+1} \leq \alpha_{\tau(i)+1}.
    \end{cases}
  \end{align*}
\end{itemize}

\subsection{Type A.1}

We consider the $\imath$quantum group $\U^\imath$ of type A.1.
In this case, $\tau$ is the identity.
We take parameters $\varsigma_i,\kappa_i$ in \eqref{B_i} as $\varsigma_i=q^{-1},\kappa_i=0$ for any $i$. 
\begin{figure}[h]
$$
\vcenter{\xymatrix@R=2ex{
  && \circ \ar@{-}[dll]_<{0} \ar@{-}[drr] \\
\circ \ar@{-}[r]_<{1} & \circ \ar@{-}[r]_<{2} & \cdots \ar@{-}[r]_>{n-2} & \circ \ar@{-}[r]_>{n-1} & \circ
}}
$$
\caption{Satake diagram of type A.1.}
\end{figure}
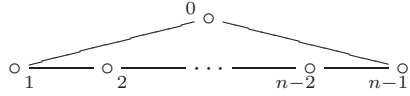
    The $\imath$crystal graph of $\B$ is given as follows:
    $$
    
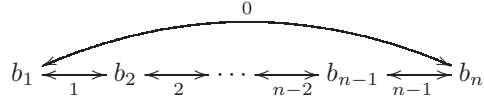
\xymatrix{
      b_1 \ar@{<->}[r]_1 & b_2 \ar@{<->}[r]_2 & \cdots \ar@{<->}[r]_{n-2} & b_{n-1} \ar@{<->}[r]_{n-1} & b_n \ar@{<->}@/_20pt/[llll]_0
    }
    $$
Here $b_i=v_i+q^{-1}\mathcal{L}$.
This type is twisted \eqref{twisted}, and the quantum $K$-matrix for $\W$ is given by
\[
K_\W(x)=\sum_jE_{n+1-j,j},
\]
which is a $\U^\imath$-linear map from $\W(x)$ to $\W^\vee(x^{-1})$.

Let $s \geq 1$, and consider KR crystals $\B_s(x) = \{ x^db_\alpha \}$ and $\B_s^\vee(x^{-1}) = \{ x^db^\vee_\alpha \}$.
We shall define a map
\begin{equation} \label{K for A1}
  \K: \B_s(x) \rightarrow \B_s^\vee(x^{-1}),\quad x^db_\alpha\mapsto x^{d+I(\alpha)}b^\vee_{\alpha'},
\end{equation}
where
\begin{align*}
  \begin{split}
    &I(\alpha) = 2\left\lfloor \frac{\alpha_1}{2} \right\rfloor = \alpha_1 - \theta(\alpha_1), \\
    &\alpha' = (\alpha'_1,\ldots,\alpha'_n),\quad\alpha'_i = \alpha_{i+1} + \theta(\alpha_i) - \theta(\alpha_{i+1}).
  \end{split}
\end{align*}

Below, we shall show that the map $\K$ is the combinatorial $K$-matrix.
First of all, let us observe that $\K$ is a bijection.
Actually, the assignment
$$
x^d b^\vee_\alpha \mapsto x^{d+J(\alpha)} b_{\alpha''},
$$
where
\begin{align*}
  \begin{split}
    &J(\alpha) = -\alpha_n + \theta(\alpha_{n}), \\
    &\alpha''_i = \alpha_{i-1} + \theta(\alpha_{i}) - \theta(\alpha_{i-1}),
  \end{split}
\end{align*}
gives rise to the inverse map.

Next, let us fix a normalized $K$-matrix as follows.
It is easy to see that
$$
\{ \alpha \mid \beta_1(b_\alpha) = s \} = \{ \mathbf{e}_1 + (s-1)\mathbf{e}_2,\ s\mathbf{e}_2 \},
$$
and
$$
\tilde{B}_1b_{\mathbf{e}_1 + (s-1)\mathbf{e}_2} = b_{s\mathbf{e}_2},\ \tilde{B}_1b_{s\mathbf{e}_2} = b_{\mathbf{e}_1 + (s-1)\mathbf{e}_2}.
$$
Then, by \cite[Section 4.1]{Wa2}, there exists a unique $u \in \mathcal{L}_s$ such that
$$
u + q^{-1}\mathcal{L}_s = b_{\mathbf{e}_1 + (s-1)\mathbf{e}_2} + b_{s\mathbf{e}_2}
$$
and that
$$
\{ w \in \W_s \mid B_1w = [s]_1w \} = \mathbb{C}(q)u.
$$
Similarly, there exists a unique $u^\vee \in \mathcal{L}_s^\vee$ such that
$$
u^\vee + q^{-1}\mathcal{L}_s^\vee = b^\vee_{(s-1)\mathbf{e}_1 + \mathbf{e}_2} + b^\vee_{s\mathbf{e}_1},
$$
and that
$$
\{ w \in \W^\vee_s \mid B_1w = [s]_1w \} = \mathbb{C}(q)u^\vee.
$$
The above observation enables us to normalize the quantum $K$-matrix $K_{\W_s}(x)$ as
$$
K_{\W_s}(x)(u) = u^\vee.
$$

\begin{lemma}\label{lem: connectedness of KR crystal AI}
The $\imath$crystal $\B_s(x)$ is connected.
Consequently, we have $K_{\W_s}(x)(\mathcal{L}_s(x)) \subset \mathcal{L}_s^\vee(x^{-1})$.
\end{lemma}

\begin{proof}
  By \cite[Theorems 3.3.6 and 4.3.1]{Wa1}, each $b_\alpha$ is connected to $b_{2\lfloor \frac{\alpha_1}{2} \rfloor \mathbf{e}_1 + (s- 2\lfloor \frac{\alpha_1}{2} \rfloor)\mathbf{e}_2}$.
  Also, when $\alpha_1 > 2$, we have
  $$
  \tilde{B}_1 \tilde{B}_0 b_{2\left\lfloor \frac{\alpha_1}{2} \right\rfloor \mathbf{e}_1 + (s- 2\left\lfloor \frac{\alpha_1}{2} \right\rfloor)\mathbf{e}_2} = b_{(2\left\lfloor \frac{\alpha_1}{2} \right\rfloor-2)\mathbf{e}_1 + (s- 2\left\lfloor \frac{\alpha_1}{2} \right\rfloor + 1)\mathbf{e}_2 + \mathbf{e}_n}.
  $$
  Since the right-hand side is connected to $b_{(2\left\lfloor \frac{\alpha_1}{2} \right\rfloor-2)\mathbf{e}_1 + (s- 2\left\lfloor \frac{\alpha_1}{2} \right\rfloor + 2)\mathbf{e}_2}$, so is $\alpha$.
  Therefore, $\alpha$ is connected to $b_{s\mathbf{e}_2}$.
  This proves the assertion.
\end{proof}

By Lemma \ref{lem: connectedness of KR crystal AI}, the quantum $K$-matrix $K_{\W_s}(x)$ induces a $\mathbb{C}$-linear map
$$
\bar{K}_{\W_s}(x): \mathcal{L}_s(x)/q^{-1}\mathcal{L}_s(x) \rightarrow \mathcal{L}^\vee_s(x^{-1})/q^{-1}\mathcal{L}_s^\vee(x^{-1}).
$$
As before, we set $\K_{\B_s}(x) := \bar{K}_{\W_s}(x)|_{\B_s(x)}$.

\begin{lemma}\label{lem: K(alpha) AI}
Let $\alpha = \mathbf{e}_1 + (s-1)\mathbf{e}_2$. Then, we have
$$
\K_{\B_s}(x)(b_\alpha) = \K(b_\alpha)
$$
\end{lemma}

\begin{proof}
  Since
  \[
  \tilde{B}_0^2 b_{\mathbf{e}_1 + (s-1)\mathbf{e}_2} = b_{\mathbf{e}_1 + (s-1)\mathbf{e}_2},\ \tilde{B}_0b_{s\mathbf{e}_2} = 0,
  \]
  we see that
  \[
  \tilde{B}_0^2 u + q^{-1}\mathcal{L}_s = b_{\mathbf{e}_1 + (s-1)\mathbf{e}_2} = b_\alpha.
  \]
  Similarly, we have
  \begin{align*}
    \begin{split}
      \tilde{B}_0^2 u^\vee + q^{-1}\mathcal{L}_s^\vee &= \begin{cases}
        b^\vee_{s\mathbf{e}_1} & \text{ if $s$ is odd}, \\
        b^\vee_{(s-1)\mathbf{e}_1 + \mathbf{e}_2} & \text{ if $s$ is even}.
      \end{cases} \\
      &= \K(b_\alpha)
    \end{split}
  \end{align*}
  Since we normalized the quantum $K$-matrix as $K_{\W_s}(x)(u) = u^\vee$, we obtain
  \begin{align*}
    \begin{split}
      \K_{\B_s}(x)(b_\alpha) &= \K_{\B_s}(x)(\tilde{B}_0^2(u + q^{-1}\mathcal{L}_s(x))) \\
      &= \tilde{B}_0^2(u^\vee + q^{-1}\mathcal{L}_s^\vee(x^{-1})) \\
      &= \K(b_\alpha),
    \end{split}
  \end{align*}
  as desired.
  This proves the assertion.
\end{proof}

\begin{lemma}\label{lem: theta(alpha'_i) = theta(alpha_i)}
  Let $x^d b_\alpha \in \B_s(x)$, and write $\K(x^d b_\alpha) = x^{d'} b^\vee_{\alpha'}$.
  Then, for each $i \in I$, we have
  \[
    \theta(\alpha'_{i}) = \theta(\alpha_{i}).
  \]
\end{lemma}

\begin{proof}
  Noting that $n-\theta(n)$ is even for all $n \in \mathbb{Z}$, we see that the parity of $\alpha'_i = \alpha_{i+1} + \theta(\alpha_{i}) - \theta(\alpha_{i+1})$ coincides with that of $\theta(\alpha_{i})$.
  Hence, the assertion follows.
\end{proof}

\begin{proposition}
The map $\K$ given in \eqref{K for A1} is the combinatorial $K$-matrix.
\end{proposition}

\begin{proof}
  We only need to show that
  \[
  \K(x^d b_\alpha) = \K_{\B_s}(x)(x^d b_\alpha)
  \]
  for all $x^d b_\alpha \in \B_s(x)$.
  To this end, by Lemmas \ref{lem: connectedness of KR crystal AI} and \ref{lem: K(alpha) AI}, it suffices to prove that the map $\K$ commutes with $\tilde{B}_i$ for all $i \in I$.
  Below, we often use Lemma \ref{lem: theta(alpha'_i) = theta(alpha_i)} without stating explicitly at each time.

  Let $i \in I$, $x^d b_\alpha \in \B_s(x)$, and write $\K(x^d b_\alpha) = x^{d'} b^\vee_{\alpha'}$.
  We see that
  \begin{align}
    \tilde{B}_i(x^{d'} b^\vee_{\alpha'}) = x^{d'-(-1)^{\theta(\alpha_{i+1})}\delta_{i,0}} b^\vee_{\alpha' -(-1)^{\theta(\alpha_{i+1})}(\mathbf{e}_i-\mathbf{e}_{i+1})}.
  \end{align}
  Setting $\alpha'' := \alpha' -(-1)^{\theta(\alpha_{i+1})}(\mathbf{e}_i-\mathbf{e}_{i+1})$, we have for each $j \in I$, 
  \begin{align*}
    \begin{split}
      \alpha''_j &= \begin{cases}
        \alpha'_i - (-1)^{\theta(\alpha_{i+1})} & \text{ if } j = i, \\
        \alpha'_{i+1} + (-1)^{\theta(\alpha_{i+1})} & \text{ if } j = i+1, \\
        \alpha'_j & \text{ otherwise},
      \end{cases} \\
      &= \begin{cases}
        \alpha_{i+1} + \theta(\alpha_{i}) - \theta(\alpha_{i+1}) - (-1)^{\theta(\alpha_{i+1})} & \text{ if } j = i, \\
        \alpha_{i+2} + \theta(\alpha_{i+1}) - \theta(\alpha_{i+2}) + (-1)^{\theta(\alpha_{i+1})} & \text{ if } j = i+1, \\
        \alpha_{j+1} + \theta(\alpha_{j}) - \theta(\alpha_{j+1}) & \text{ otherwise},
      \end{cases} \\
      &= \begin{cases}
        \alpha_{i+1} + \theta(\alpha_{i}) + \theta(\alpha_{i+1}) -1 & \text{ if } j = i, \\
        \alpha_{i+2} - \theta(\alpha_{i+1}) - \theta(\alpha_{i+2}) + 1 & \text{ if } j = i+1, \\
        \alpha_{j+1} + \theta(\alpha_{j}) - \theta(\alpha_{j+1}) & \text{ otherwise}.
      \end{cases}
    \end{split}
  \end{align*}
  On the other hand, we have
  \begin{align}
    \tilde{B}_i(x^d b_\alpha) = x^{d+(-1)^{\theta(\alpha_{i})}\delta_{i,0}} b_{\alpha + (-1)^{\theta(\alpha_{i})}(\mathbf{e}_i-\mathbf{e}_{i+1})}.
  \end{align}
  Let us write
  \[
    \K(\tilde{B}_i(x^d b_\alpha)) = x^{d'''} b^\vee_{\alpha'''}.
  \]
  Then, we have
  \[
      d''' = d + (-1)^{\theta(\alpha_{i})}\delta_{i,0} + I(\alpha + (-1)^{\theta(\alpha_{i})}(\mathbf{e}_i-\mathbf{e}_{i+1})),
  \]
  and for each $j \in I$,
  \begin{align*}
    \begin{split}
      \alpha'''_j = \begin{cases}
        \alpha_i + (-1)^{\theta(\alpha_{i})} + \theta(\alpha_{i-1}) - \theta(\alpha_{i}+(-1)^{\theta(\alpha_{i})}) & \text{ if } j = i-1, \\
        \alpha_{i+1} - (-1)^{\theta(\alpha_{i})} + \theta(\alpha_{i}+(-1)^{\theta(\alpha_{i})}) - \theta(\alpha_{i+1}-(-1)^{\theta(\alpha_{i})}) & \text{ if } j = i, \\
        \alpha_{i+2} + \theta(\alpha_{i+1}-(-1)^{\theta(\alpha_{i})}) - \theta(\alpha_{i+2}) & \text{ if } j = i+1, \\
        \alpha_{j+1} + \theta(\alpha_{j}) - \theta(\alpha_{j+1}) & \text{ otherwise}.
      \end{cases}
    \end{split}
  \end{align*}
  Now, it is an easy exercise to verify that $d''' = d'-(-1)^{\theta(\alpha_{i+1})}\delta_{i,0}$ and that $\alpha'''_j = \alpha''_j$ for all $j \in I$.
  Thus, we see that $\K$ commutes with $\tilde{B}_i$, and hence, the proof completes.
\end{proof}
\begin{example} \label{ex:reflectionA1}
Set $n=3$. Below is the graphical presentation when we apply the both hand sides of the combinatorial reflection equation 
\eqref{comb RE} on $x^0b_{(1,2,2)}\ot y^0b_{(2,1,0)}\in\B_5(x)\ot\B_3(y)$.
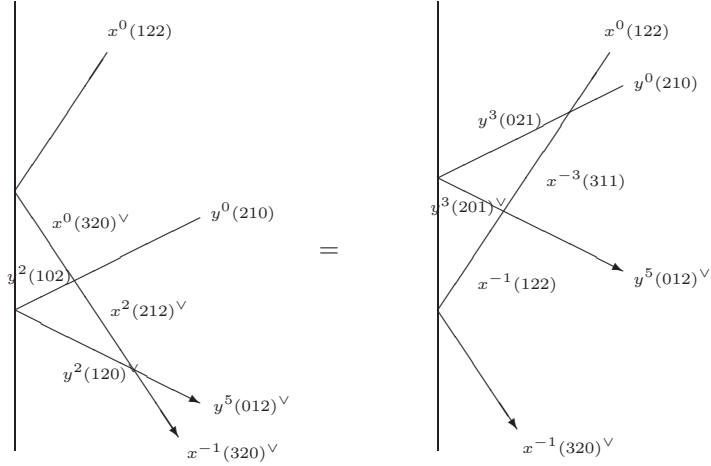
\[
\begin{picture}(300,185)(-32,5)
\put(0,20){
\put(0,-8){\line(0,1){170}}
\put(0,15){\put(0,30){\line(2,1){70}}\put(0,30){\vector(2,-1){70}}}
\put(0,0){\put(0,90){\line(2,3){35}}\put(0,90){\vector(2,-3){62}}}
\put(35,149){\tiny{$x^0(122)$}}
\put(-3,56){\tiny{$y^2(102)$}}
\put(74,80){\tiny{$y^0(210)$}}
\put(14,76){\tiny{$x^0(320)^\vee$}}
\put(36,43){\tiny{$x^2(212)^\vee$}}
\put(18,19){\tiny{$y^2(120)^\vee$}}
\put(75,7){\tiny{$y^5(012)^\vee$}}
\put(65,-11){\tiny{$x^{-1}(320)^\vee$}}
}
\put(115,85){$=$}
\put(160,20){
\put(0,-8){\line(0,1){170}}
\put(0,15){\put(0,30){\line(2,3){65}}\put(0,30){\vector(2,-3){30}}}
\put(0,5){\put(0,90){\line(2,1){70}}\put(0,90){\vector(2,-1){70}}}
\put(63,149){\tiny{$x^0(122)$}}
\put(74,129){\tiny{$y^0(210)$}}
\put(15,115){\tiny{$y^3(021)$}}
\put(41,92){\tiny{$x^{-3}(311)$}}
\put(-3,82){\tiny{$y^3(201)^\vee$}}
\put(15,53){\tiny{$x^{-1}(122)$}}
\put(74,55){\tiny{$y^5(012)^\vee$}}\put(32,-9){\tiny{$x^{-1}(320)^\vee$}}
}
\end{picture}
\]
\end{example}

\subsection{Type A.3}

We consider the $\imath$quantum group $\U^\imath$ of type A.3 in the quasi-split case.
In this case, $\tau$ is given by $\tau(i) = -i \pmod{n}$. We set $n'=\lfloor\frac{n-1}2\rfloor$. 
We take parameters $\varsigma_i,\kappa_i$ in \eqref{B_i} as $\varsigma_0=q^{-1},
\varsigma_{n'+1}=q^{-1}$ if $n$ is even, $\varsigma_{n'+1}=q$ if $n$ is odd, 
$\varsigma_i=1$ otherwise, and $\kappa_i=0$ for any $i$. 
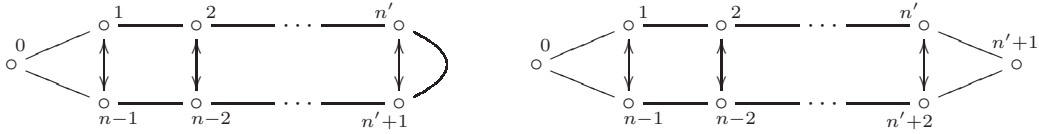
\begin{figure}[h]
$$
\vcenter{\xymatrix@R=1ex{
& \circ \ar@{-}[r]^<{1} \ar@{<->}[dd] & \circ \ar@{-}[r]^<{2} \ar@{<->}[dd] & \cdots \ar@{-}[r]^>{n'} & \circ \ar@{-}@/^18pt/[dd] \ar@{<->}[dd] \\
\circ \ar@{-}[ur]^<{0} \ar@{-}[dr] & \\
& \circ \ar@{-}[r]_<{n-1} & \circ \ar@{-}[r]_<{n-2} & \cdots \ar@{-}[r]_>{n'+1} & \circ \\
}}
\qquad\qquad
\vcenter{\xymatrix@R=1ex{
& \circ \ar@{-}[r]^<{1} \ar@{<->}[dd] & \circ \ar@{-}[r]^<{2} \ar@{<->}[dd] & \cdots \ar@{-}[r]^>{n'} & \circ \ar@{-}[dr]^>{n'+1} \ar@{<->}[dd] \\
\circ \ar@{-}[ur]^<{0} \ar@{-}[dr] &&&&& \circ \\
& \circ \ar@{-}[r]_<{n-1} & \circ \ar@{-}[r]_<{n-2} & \cdots \ar@{-}[r]_>{n'+2} & \circ \ar@{-}[ur] \\
}}
$$
\caption{Satake diagrams of type A.3 when $n$ is odd (left) and even (right).}
\end{figure}
    The $\imath$crystal graph of $\B$ when $n$ is even is given as follows:
    $$
    \xymatrix{
      b_1 \ar[r]^1 \ar@{<->}[d]_0 & b_2 \ar[r]^2 & \cdots \ar[r]^{n'} & b_{n'+1} \ar@{<->}[d]^{n'+1} \\
      b_{n} \ar[r]^1 & b_{n-1} \ar[r]^2 & \cdots \ar[r]^{n'} & b_{n'+2} \\
    }
    $$
Here $b_i=v_i+q^{-1}\mathcal{L}$. This type is untwisted \eqref{untwisted}, and the quantum $K$-matrix for $\W$ is given by
\[
K_\W(x)=\sum_jE_{1-j,j},
\]
which is a $\U^\imath$-linear map from $\W(x)$ to $\W(x^{-1})$.

Note that $\B_s=\{b_\alpha\,|\,\alpha=(\alpha_i)_{1\le i\le n},\alpha_i\ge0,\sum_i\alpha_i=s\}$.
To describe the combinatorial $K$-matrix $\K_{\B_s}(x)$, we need to divide
the set $\{\alpha=(\alpha_i)_{1\le i\le n}\mid b_\alpha\in\B_s\}$ into $n$ subsets as follows.
If $(\alpha_i)$ satisfies 
\begin{align*}
&\alpha_1+\alpha_2+\cd+\alpha_i<\alpha_{n-i}+\cd+\alpha_{n-1},\alpha_2+\cd+\alpha_i<\alpha_{n-i}+\cd+\alpha_{n-2},\cd,\alpha_i<\alpha_{n-i},\\
&\text{and }\alpha_{i+1}\ge \alpha_{n-i-1},\alpha_{i+1}+\alpha_{i+2}\ge \alpha_{n-i-2}+\alpha_{n-i-1},\cd,
\alpha_{i+1}+\cd+\alpha_{n'}\ge \alpha_{n-n'}+\cd+\alpha_{n-i-1},\\
&\text{and }\alpha_1+\cd+\alpha_i+\alpha_{n-i}+\cd+\alpha_n:\text{odd}
\end{align*}
for some $i$ such that $0\le i\le n'$, then say Case $(2i+1)$ holds.
If $(\alpha_i)$ satisfies 
\begin{align*}
&\alpha_1+\alpha_2+\cd+\alpha_i\le \alpha_{n-i}+\cd+\alpha_{n-1},\alpha_2+\cd+\alpha_i\le \alpha_{n-i}+\cd+\alpha_{n-2},\cd,\alpha_i\le \alpha_{n-i},\\
&\text{and }\alpha_{i+1}>\alpha_{n-i-1},\alpha_{i+1}+\alpha_{i+2}>\alpha_{n-i-2}+\alpha_{n-i-1},\cd,
\alpha_{i+1}+\cd+\alpha_{n'}>\alpha_{n-n'}+\cd+\alpha_{n-i-1},\\
&\text{and }\alpha_1+\cd+\alpha_i+\alpha_{n-i}+\cd+\alpha_n:\text{even}
\end{align*}
for some $i$ such that $0\le i\le n'$, then say Case $(2i+2)$ holds. Note that when 
both $n$ and $s$ are odd, then Case ($n+1$) does not occur.
We use this case division only when $s$ is odd.

\begin{remark}
We list the case divisions when $n$ is small and $s$ is odd.

$n=3$
\[
\begin{array}{lll}
\text{Case (1)}\;&\alpha_1\ge \alpha_2,\;&\alpha_3\text{ is odd},\\
\text{Case (2)}\;&\alpha_1>\alpha_2,\;&\alpha_3\text{ is even},\\
\text{Case (3)}\;&\alpha_1<\alpha_2.
\end{array}
\]

$n=4$
\[
\begin{array}{llll}
\text{Case (1)}\;&\alpha_1\ge \alpha_3,&\alpha_4\text{ is odd},\\
\text{Case (2)}\;&\alpha_1>\alpha_3,&\alpha_4\text{ is even},\\
\text{Case (3)}\;&\alpha_1<\alpha_3,&\alpha_1+\alpha_3+\alpha_4\text{ is odd},\\
\text{Case (4)}\;&\alpha_1\le \alpha_3,&\alpha_1+\alpha_3+\alpha_4\text{ is even}.
\end{array}
\]

$n=5$
\[
\begin{array}{lllll}
\text{Case (1)}\;&\alpha_1\ge \alpha_4,&\alpha_1+\alpha_2\ge \alpha_3+\alpha_4,,&\alpha_5\text{ is odd},\\
\text{Case (2)}\;&\alpha_1>\alpha_4,&\alpha_1+\alpha_2>\alpha_3+\alpha_4,&\alpha_5\text{ is even},\\
\text{Case (3)}\;&\alpha_1<\alpha_4,&\alpha_2\ge \alpha_3,&\alpha_1+\alpha_4+\alpha_5\text{ is odd},\\
\text{Case (4)}\;&\alpha_1\le \alpha_4,&\alpha_2>\alpha_3,&\alpha_1+\alpha_4+\alpha_5\text{ is even},\\
\text{Case (5)}\;&\alpha_1+\alpha_2<\alpha_3+\alpha_4,&\alpha_2<\alpha_3.
\end{array}
\]
\end{remark}

\begin{lemma} \label{lem:O1}
Suppose $s$ is odd. Then, for any $\alpha$ such that $b_\alpha\in\B_s$, it belongs to one and only one Case $(i)$ for some $1\le i\le n$. 
\end{lemma}
\begin{proof}
For $n=3,4$ cases, it is readily checked. 

We show one can reduce the proof of the $n$ case to the $n-2$  case by taking the $n=5$ case for example.
Suppose $\alpha_1>\alpha_4$. Then, under this condition, only Case (1),(2),(5) can occur. Note that in Case (5),
the second inequality automatically satisfied if the first one holds. Thus, the fact that one and only one
case of Case (1),(2),(5) occurs is deduced from the $n=3$ case by considering the case of $(\alpha_1+\alpha_2,\alpha_3+\alpha_4,\alpha_5)$.
Next, suppose $\alpha_1<\alpha_4$. Then, under this condition, only Case (3),(4),(5) can occur. Note that in Case (5),
the first inequality automatically satisfied if the second one holds. Thus, the fact that one and only one
case of Case (3),(4),(5) occurs is deduced again from the $n=3$ case by considering the case of
$(\alpha_2,\alpha_3,\alpha_1+\alpha_4+\alpha_5)$. Finally, suppose $\alpha_1=\alpha_4$. 
Then, under this condition, only Case (1),(4),(5) can occur, and the fact that one and only one
case occurs is deduced again from the $n=3$ case by considering the case of
$(\alpha_2,\alpha_3,\alpha_5)$.
\end{proof}

\begin{proposition}\label{prop: connectedness AIII}
The $\imath$crystal $\B_s$ is connected.
\end{proposition}

This proposition follows from the next three lemmas.

\begin{lemma} \label{lem:O2}
Suppose $i\ne0$.
\begin{itemize} 
\item[(1)] If $\tilde{B}_ib_\alpha=b_{\alpha'}$, then 
$\tilde{B}_{n-i}b_{\alpha'}=b_\alpha$.
\item[(2)] If $\tilde{B}_{i}b_\alpha=\frac1{\sqrt{2}}(b_{\alpha'}+b_{\alpha''})$, then
$\tilde{B}_{i-1}b_{\alpha'}=\frac1{\sqrt{2}}b_\alpha$ and
$\tilde{B}_{i-1}b_{\alpha''}=\frac1{\sqrt{2}}b_\alpha$.
It happens only when $n$ is even and $i=n'+1$.
\item[(3)] If $\tilde{B}_ib_{\alpha}=\frac1{\sqrt{2}}b_{\alpha'}$, then
$\tilde{B}_{i+1}b_{\alpha'}=\frac1{\sqrt{2}}(b_{\alpha}+b_{\alpha''})$ with some other $\alpha''$.
It happens only when $n$ is even and $i=n'$.
\end{itemize}
\end{lemma}
\begin{proof}
The assertions follow from \cite[Sections 4.2 and 4.3]{Wa2}.
\end{proof}

\begin{lemma} \label{lem:O3}
Any element $b_\alpha\in\B_s$ is generated by $\tilde{B}_i$ $(i\ne0)$ from $b_{\alpha'}$ where
$\alpha'=(\sum_{j=1}^{n-1}\alpha_j-\gamma_1)\eb_1+(\alpha_n+\gamma_1)\eb_n$.
$\gamma_1$ is defined inductively by
\[
\gamma_{n'}=(\alpha_{n-n'}-\alpha_{n'})_+,\quad
\gamma_j=(\gamma_{j+1}+\alpha_{n-j}-\alpha_j)_+\;\text{for }j=n'-1,n'-2,\ldots,1.
\]
\end{lemma}
\begin{proof}
We give the proof for $n=5$ only. Note that $n'=2,n-n'=3$ in this case. 
Applying $\tilde{B}_3$, we have 
\[
\tilde{B}_3^{\alpha_3}b_{\alpha}\propto b_{\alpha'}+\cdots,\quad
\alpha'=(\alpha_1,\alpha_2+\alpha_3-\gamma_2,0,\alpha_4+\gamma_2,\alpha_5).
\]
We used the symbol $\propto$ since the factor $\frac1{\sqrt{2}}$ may occur
upon application. Other terms may also appear. Similarly, we continue the calculation as
\begin{align*}
&\tilde{B}_4^{\alpha_2+\alpha_3-\gamma_2+\gamma_1}b_{\alpha'}\propto 
b_{\alpha''}+\cdots,\quad
\alpha''=(\alpha_1+\alpha_2+\alpha_3-\gamma_2,0,0,\alpha_4+\gamma_2-\gamma_1,
\alpha_5+\gamma_1),\\
&(\tilde{B}_3\tilde{B}_2)^{\alpha_4+\gamma_2-\gamma_1}b_{\alpha''}
\propto b_{\alpha'''}+\cdots,\quad \alpha'''=(\alpha_1+\alpha_2+\alpha_3-\gamma_2,
\alpha_4+\gamma_2-\gamma_1,0,0,\alpha_5+\gamma_1),\\
&\tilde{B}_4^{\alpha_4+\gamma_2-\gamma_1}b_{\alpha'''}
\propto b_{\alpha''''}+\cdots,\quad
\alpha''''=(\alpha_1+\alpha_2+\alpha_3+\alpha_4-\gamma_1,0,0,0,\alpha_5+\gamma_1).
\end{align*}
In view of Lemma \ref{lem:O2}, we find that $b_\alpha$ is generated from $b_{\alpha''''}$.
\end{proof}

\begin{lemma} \label{lem:O4}
\begin{itemize}
\item[(1)] 
If $a$ is even, $b_{(s-a)\eb_1+a\eb_n}$ and $b_{(s-a-1)\eb_1+(a+1)\eb_n}$ are connected
by $\tilde{B}_0$.
\item[(2)] 
If $n>3$ and $a$ is even, $b_{(s-a)\eb_1+a\eb_n}$ and $b_{(s-a-2)\eb_1+(a+2)\eb_n}$ are 
connected by applications of $\tilde{B}_i$'s.
\item[(3)] 
If $n=3$ and $a$ is odd, $b_{(s-a-1)\eb_1+(a+1)\eb_3}$ is represented as a 
linear combination of vectors which are generated from $b_{(s-a)\eb_1+a\eb_3}$ 
by applying $\tilde{B}_i$'s.
\end{itemize}
\end{lemma}
\begin{proof}
(1) is immediate. For (2), consider two vectors $b_{\eb_1+(s-a-2)\eb_2+\eb_{n-1}+a\eb_n}$
and $b_{(s-a-2)\eb_2+\eb_{n-1}+(a+1)\eb_n}$ which are connected by $\tilde{B}_0$. 
On the other hand, by Lemma \ref{lem:O3}, the former can be connected with 
$b_{(s-a)\eb_1+a\eb_n}$ and the latter $b_{(s-a-2)\eb_1+(a+2)\eb_n}$. Note that
cases (2),(3) of Lemma \ref{lem:O2} do not occur during the applications.

To prove (3), set $c=s-a$ and divide the cases whether $c$ is even or odd. Suppose 
$c$ is even. Then we have
\[
\tilde{B}_2^2\tilde{B}_0\tilde{B}_1^{c/2+1}b_{c\eb_1+a\eb_3}
=\frac1{\sqrt{2}}(b_{c/2\eb_1+(c/2-1)\eb_2+(a+1)\eb_3}
+b_{(c/2-1)\eb_1+(c/2-1)\eb_2+(a+2)\eb_3}).
\]
Applying $\tilde{B}_0$, we obtain a linear combination of the same two vectors, but
with different coefficients. Hence, one obtains $b_{c/2\eb_1+(c/2-1)\eb_2+(a+1)\eb_3}$
as a linear combination of vectors which are generated from $b_{c\eb_1+a\eb_3}$ 
by applying $\tilde{B}_i$'s. Applying $\tilde{B}_2^{c/2-1}$, one obtains 
$b_{(c-1)\eb_1+(a+1)\eb_3}$. 
Next suppose $c$ is odd. We have
\begin{align*}
&\tilde{B}_1^{(c-1)/2}b_{c\eb_1+a\eb_3}=b_{(c+1)/2\eb_1+(c-1)/2\eb_2+a\eb_3},\\
&\tilde{B}_2\tilde{B}_1^{(c+1)/2}b_{c\eb_1+a\eb_3}
=\frac1{\sqrt{2}}(b_{(c+1)/2\eb_1+(c-1)/2\eb_2+a\eb_3}
+b_{(c-1)/2\eb_1+(c-1)/2\eb_2+(a+1)\eb_3}).
\end{align*}
Hence, one obtains $b_{(c-1)/2\eb_1+(c-1)/2\eb_2+(a+1)\eb_3}$
as a linear combination of vectors which are generated from $b_{c\eb_1+a\eb_3}$ 
by applying $\tilde{B}_i$'s. Applying $\tilde{B}_2^{(c-1)/2}$ further, one obtains 
$b_{(c-1)\eb_1+(a+1)\eb_3}$. 
\end{proof}

\begin{lemma} \label{lem:O5}
Under a suitable normalization, we have
\[
\K_{\B_s}(x)(b_{s\mathbf{e}_n}) = \begin{cases}
  b_{\mathbf{e}_1 + (s-1)\mathbf{e}_n} & \text{ if $s$ is odd}, \\
  b_{s\mathbf{e}_n} & \text{ if $s$ is even}.
\end{cases}
\]
\end{lemma}

\begin{proof}
Let us first consider the case when $n$ is odd.
By weight consideration, we can normalize the quantum $K$-matrix as
$$
K_{\W_s}(x)(v_{s\mathbf{e}_{n'+1}}) = v_{s\mathbf{e}_{n'+1}}.
$$
Hence, we have
$$
\K_{\B_s}(x)(b_{s\mathbf{e}_{n'+1}}) = b_{s\mathbf{e}_{n'+1}}.
$$
Applying $\sqrt{2} \tilde{B}_{n-1}^s \tilde{B}_{n-2}^s \cdots \tilde{B}_{n'+2}^s \tilde{B}_{n'+1}^s$ on the both sides, we obtain
\begin{align}
  \K_{\B_s}(x)(b_{\mathbf{e}_1 + (s-1)\mathbf{e}_n} + b_{s\mathbf{e}_n}) = b_{\mathbf{e}_1 + (s-1)\mathbf{e}_n} + b_{s\mathbf{e}_n}. \label{eq: comb K on special vector AIII n odd}
\end{align}

When $s$ is odd, applying $\tilde{B}_0$ on the both sides of \eqref{eq: comb K on special vector AIII n odd}, we have
\[
\K_{\B_s}(x)(xb_{s\mathbf{e}_1} + x^{-1}b_{\mathbf{e}_1 + (s-1)\mathbf{e}_n}) = x^{-1} b_{s\mathbf{e}_n} + x b_{\mathbf{e}_1 + (s-1)\mathbf{e}_n}.
\]
Comparing this with \eqref{eq: comb K on special vector AIII n odd}, the assertion follows.

Similarly, when $s$ is even, applying $\tilde{B}_0^2$ on the both sides of \eqref{eq: comb K on special vector AIII n odd}, we have
\[
\K_{\B_s}(x)(b_{\mathbf{e}_1 + (s-1)\mathbf{e}_n}) = b_{\mathbf{e}_1 + (s-1)\mathbf{e}_n}.
\]
Comparing this with identity \eqref{eq: comb K on special vector AIII n odd}, the assertion follows.

Next, let us consider the case when $n$ is even.
As in type A.1, from the consideration of $\beta_{n'+1}$, we can normalize the quantum $K$-matrix as
\[
\K_{\B_s}(b_{\mathbf{e}_{n'+1} + (s-1)\mathbf{e}_{n'+2}} + b_{s\mathbf{e}_{n'+2}}) = b_{\mathbf{e}_{n'+1} + (s-1)\mathbf{e}_{n'+2}} + b_{s\mathbf{e}_{n'+2}}.
\]
Applying $\tilde{B}_{n-1}^s \tilde{B}_{n-2}^s \cdots \tilde{B}_{n'+2}^s$ on both sides, we obtain
\begin{align*}
  \K_{\B_s}(x)(b_{\mathbf{e}_1 + (s-1)\mathbf{e}_n} + b_{s\mathbf{e}_n}) = b_{\mathbf{e}_1 + (s-1)\mathbf{e}_n} + b_{s\mathbf{e}_n}.
\end{align*}
Now, as in the $n$ being odd case, applying $\tilde{B}_0$ when $s$ is odd or $\tilde{B}_0^2$ when $s$ is even, we obtain the desired identity.
\end{proof}

\begin{proposition}
The combinatorial $K$-matrix $\K_{\B_s}(x)$ is given by
\begin{equation}  \label{K_AIII}
\K_{\B_s}(x)(x^d b_\alpha)=x^{d+I(\alpha)}
\left\{\begin{array}{ll}
b_\alpha&(s\text{ is even})\\
b_{\alpha+(-1)^{i-1}(\eb_{i'}-\eb_{n-i'+1})}&(s\text{ is odd})
\end{array}\right.,
\end{equation} 
where we assume for $\alpha$ Case $(i)$ $(1\le i\le n)$ holds and $i'=\lfloor\frac{i+1}2\rfloor$. The energy function $I(\alpha)$ is given by
\[
I(\alpha)=-2\lfloor\frac{\alpha_n+\gamma_1+\chi(\text{$s$ is even})}2\rfloor,
\]
where $\gamma_1$ was defined in Lemma \ref{lem:O3}.
\end{proposition}

\begin{proof}
In view of Proposition \ref{prop: connectedness AIII} and Lemma \ref{lem:O5}, it suffices to show that the right hand side of 
\eqref{K_AIII} commutes with $\tilde{B}_i$ for any $i$, which is to be proved next.
\end{proof}

\begin{lemma} \label{lem: commutativity AIII}
The right hand side of \eqref{K_AIII} commutes with $\tilde{B}_i$ for any $i$.
\end{lemma}

\begin{proof}
We let $\K(x)$ denote the right hand side of \eqref{K_AIII}, and show $\K(x)$ comments with $\tilde{B}_i$ for any $i$.
If $s$ is even, $\K(1)$ is the identity. Hence, the commutativity is trivial except 
the case of $\tilde{B}_0$ where there is a change in the power of $x$. Since this case is close to when $s$ is odd, we omit its proof.

We assume $s$ is odd below. To reduce the number of cases to handle, we first note the following facts that can be checked easily from previous results.
\begin{itemize}
\item[(i)] 
If $n$ is odd and Case $(n)$ holds for $\alpha$, then $\K(1)(b_\alpha)=b_\alpha$.
\item[(ii)]
Suppose $i<n'$ if $n$ is even. 
If Case $(2i+1)$ (resp. $(2i+2)$) holds for $\alpha$, then for $\K(1)(b_\alpha)$ Case $(2i+2)$ (resp. $(2i+1)$) holds.
\item[(iii)]
If Case $(2i+1)$ or $(2i+2)$ holds for $\alpha$, then $I(\alpha)=-2\lfloor\frac{(\alpha_n+\cdots+\alpha_{n-i})-(\alpha_i+\cdots+\alpha_1)}2\rfloor$.
\item[(iv)]
$\tilde{B}_i$($i\ne0$) does not change the energy function, i.e., $I(\tilde{B}_i\alpha)=
I(\alpha)$.
\end{itemize}
From these facts, one verifies that $\K(x^{-1})\K(x)=\mathrm{id}$. Regarding them and Lemma \ref{lem:O2}(1), one notices
that we can restrict the cases for the proof of the commutativity of $\tilde{B}_i$ and $\K(x)$ to the following ones.
\begin{itemize}
\item[(a)]
$\tilde{B}_0$ and Case (1) holds.
\item[(b)]
$\tilde{B}_0$ and Case $(2i+1)$ holds for $0<i\le n'$.
\item[(c)]
$\tilde{B}_{i+1}$ and Case $(2i+1)$ holds for $0\le i<n'$.
\item[(d)]
$\tilde{B}_j$ and Case $(2i+1)$ holds for $j\ne i+1,1\le j\le n',0\le i\le n'$.
\item[(e)]
$\tilde{B}_{n'+1}$ and Case $(2n'+1)$ holds for $n$ even.
\item[(f)]
$\tilde{B}_{n'+1}$ and Case $(2i+1)$ holds for $0\le i<n'$ and $n$ even.
\item[(g)]
$\tilde{B}_{n'+1}$ and Case $(2n'-1)$ holds for $n$ odd.
\item[(h)]
$\tilde{B}_{n'+1}$ and Case $(2i+1)$ holds for $0\le i<n'-1$ and $n$ odd.
\end{itemize}
We prove only in the cases of (a),(c) and (g).

Let us show (a). Suppose $\alpha_n$ is odd. Then we have $\K(x)(b_\alpha)=x^{-\alpha_n+1}
b_{\alpha'},\tilde{B}_0b_\alpha=x^{-1}b_\alpha'$, where $\alpha'=\alpha+\eb_1-\eb_n$. Since
Case (2) holds for $\alpha'$, $\K(x)(\tilde{B}_0b_\alpha)=\tilde{B}_0\K(x)(b_\alpha)=
x^{-\alpha_n}b_\alpha$. Note that $\K(x)(b_\alpha)$ belongs to $\B_s(x^{-1})$. The case 
$\alpha_n$ is even is similar.

Next we show (c). Because of the fact (iv) above, one can ignore the dependence of $x$.
We have $\K(1)(b_\alpha)=b_{\alpha+\eb_{i+1}-\eb_{n-i}}$, and 
\[
\tilde{B}_{i+1}b_\alpha=\left\{
\begin{array}{ll}
b_{\alpha-\eb_{i+1}+\eb_{i+2}}\quad&(\alpha_{i+1}>\alpha_{n-i-1})\\
b_{\alpha+\eb_{n-i-1}-\eb_{n-i}}\quad&(\alpha_{i+1}=\alpha_{n-i-1}).
\end{array}\right.
\]
For the former case, Case $(2i+1)$ holds, whereas for the latter, Case $(2i+3)$ holds. 
In either case, $\K(1)(\tilde{B}_{i+1}b_\alpha)=b_{\alpha+\eb_{i+2}-\eb_{n-i}}$, which agrees with
$\tilde{B}_{i+1}\K(1)(b_\alpha)$.

Finally, we show (g). In this case, we have $\K(1)(b_\alpha)=b_{\alpha+\eb_{n'}-\eb_{n'+2}}$ and
$\tilde{B}_{n'+1}b_\alpha=b_{\alpha+\eb_{n'}-\eb_{n'+1}}$ since Case $(2n'-1)$ for $\alpha$ implies
$\alpha_{n'}\ge\alpha_{n'+1}$. Since Case $(2n'-1)$ holds also for $\tilde{B}_{n'+1}b_{\alpha}$,
one finds $\tilde{B}_{n'+1}\K(1)(b_\alpha)=\K(1)(\tilde{B}_{n'+1}b_\alpha)=b_{\alpha+2\eb_{n'}-\eb_{n'+1}-\eb_{n'+2}}$.
\end{proof}
\begin{example} \label{ex:reflectionA3}
Set $n=3$. Below is the graphical presentation when we apply the both hand sides of the combinatorial reflection equation 
\eqref{comb RE} on $x^0b_{(2,2,1)}\ot y^0b_{(2,1,0)}\in\B_5(x)\ot\B_3(y)$.
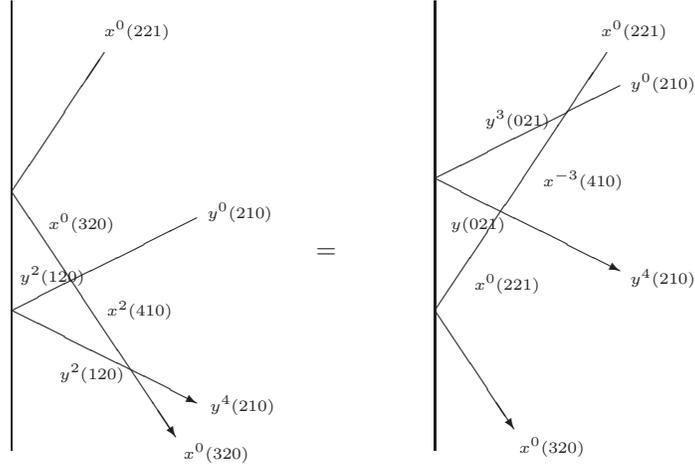
\[
\begin{picture}(300,185)(-32,5)
\put(0,20){
\put(0,-8){\line(0,1){170}}
\put(0,15){\put(0,30){\line(2,1){70}}\put(0,30){\vector(2,-1){70}}}
\put(0,0){\put(0,90){\line(2,3){35}}\put(0,90){\vector(2,-3){62}}}
\put(35,149){\tiny{$x^0(221)$}}
\put(3,56){\tiny{$y^2(120)$}}
\put(74,80){\tiny{$y^0(210)$}}
\put(14,76){\tiny{$x^0(320)$}}
\put(36,43){\tiny{$x^2(410)$}}
\put(18,19){\tiny{$y^2(120)$}}
\put(75,7){\tiny{$y^4(210)$}}
\put(65,-11){\tiny{$x^0(320)$}}
}
\put(115,85){$=$}
\put(160,20){
\put(0,-8){\line(0,1){170}}
\put(0,15){\put(0,30){\line(2,3){65}}\put(0,30){\vector(2,-3){30}}}
\put(0,5){\put(0,90){\line(2,1){70}}\put(0,90){\vector(2,-1){70}}}
\put(63,149){\tiny{$x^0(221)$}}
\put(74,129){\tiny{$y^0(210)$}}
\put(19,115){\tiny{$y^3(021)$}}
\put(41,92){\tiny{$x^{-3}(410)$}}
\put(6,76){\tiny{$y(021)$}}
\put(15,53){\tiny{$x^0(221)$}}
\put(74,55){\tiny{$y^4(210)$}}\put(32,-9){\tiny{$x^0(320)$}}
}
\end{picture}
\]
\end{example}

\subsection{Type A.4}
We consider the $\imath$quantum group of type A.4. In this case, $n$ should be even and we set $n'=n/2$. 
\begin{figure}[h]
$$
\vcenter{\xymatrix@R=1ex{
& \circ \ar@{-}[r]^<{1} \ar@{<->}[rdd] & \circ \ar@{-}[rd]^<{2} \ar@{<->}[ldd] \\
\circ \ar@{-}[ur]^<{0} \ar@{<->}[rrr] \ar@{-}[dr] & & & \circ \\
& \circ \ar@{-}[r]_<{5} & \circ \ar@{-}[ur]_<{4}_>{3} \\
}}
$$
\caption{Satake diagram of type A.4 when $n=6$.}
\end{figure}
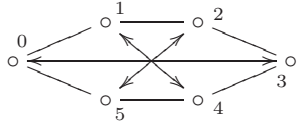
This type is twisted \eqref{twisted}, and the quantum $K$-matrix for $\W$ is given by
\[
K_\W(x)=\sum_jx^{\chi(j>n')}E_{j+n',j},
\]
which is a $\U^\imath$-linear map from $\W(x)$ to $\W^\vee(x^{-1})$.
\begin{proposition}
Define $\K:\B_s(x)\rightarrow \B^\vee_s(x^{-1})$ by $\K(x^db_\alpha)=x^{d+I(\alpha)}b_{\alpha'}^\vee$ where
\begin{align*}
&I(\alpha)=\min(\alpha_1,\alpha_{n'+1})-\sum_{j=n'+1}^n\alpha_j, \\
&\alpha'_i=\alpha_{i+1}+\alpha_{i+n'+1}+\max(\alpha_i,\alpha_{i+n'})-\alpha_i
-\max(\alpha_{i+1},\alpha_{i+n'+1}).
\end{align*}
Then, it is the combinatorial $K$-matrix $\K_{\B_s}(x)$.
\end{proposition}

We list necessary lemmas to prove this proposition below.

\begin{lemma}
We have
\[
\K(b_{s\mathbf{e}_1}) = \K_{\B_s}(x)(b_{s\mathbf{e}_1}) = b_{s\mathbf{e}_{n'+1}}^\vee.
\]
\end{lemma}

\begin{proof}
  The assertion follows from weight consideration and easy calculation.
\end{proof}

\begin{lemma}
$\K$ commutes with $\tilde{B}_i$ for any $i \in I$.
\end{lemma}
\begin{proof}
We first prove the commutativity when $x=1$.
Let $b_{\alpha} \in \B_s$. In the case of type A.4, the action of $\tilde{B}_i$ on $\B_s$ is given by the case of $a_{i,\tau(i)} = 0$. To describe it concretely, we assume $\varphi_i(b_{\alpha}) > \varphi_{i+n'}(b_{\alpha})$. In this case, it satisfies $\tilde{B}_ib_{\alpha} = b_{\alpha - \eb_i + \eb_{i+1}}$. Then, the entries of $\K\tilde{B}_ib_{\alpha}$ related to the $i$-th and $(i+1)$-th ones of $\tilde{B}_ib_{\alpha}$ are described as follows:
\[
\begin{array}{ll}
\text{$(i-1)$-th} \ & \ \alpha_{i+n'} - \alpha_{i-1} + \max(\alpha_{i-1} , \alpha_{i+n'-1})\\
\text{$i$-th} \ & \ \alpha_{i+1} + \alpha_{i+n'+1} - \max(\alpha_{i+1}+1 , \alpha_{i+n'+1}) + 1\\
\text{$(i+1)$-th} \ & \ \alpha_{i+2} + \alpha_{i+n'+2} + \max(\alpha_{i+1}+1 , \alpha_{i+n'+1}) - \alpha_{i+1} - \max(\alpha_{i+2},\alpha_{i+n'+2}) -1\\
\text{$(i+n'-1)$-th} \ & \ \alpha_{i+n'} - \alpha_{i+n'-1} + \max(\alpha_{i-1} , \alpha_{i+n'-1})\\
\text{$(i+n')$-th} \ & \ \alpha_i + \alpha_{i+1} -\alpha_{i+n'}+\alpha_{i+n'+1}- \max(\alpha_{i+1}+1 , \alpha_{i+n'+1})\\
\text{$(i+n'+1)$-th} \ & \ \alpha_{i+2} + \alpha_{i+n'+2} + \max(\alpha_{i+1}+1 , \alpha_{i+n'+1})-\alpha_{i+n'+1}-\max(\alpha_{i+2} , \alpha_{i+n'+2})\\
\end{array}
\]
To see the action of $\tilde{B}_i$ on $\K (b_{\alpha})$, we should compare $\varphi_i(\K (b_{\alpha}))$ and $\varphi_{i+n'}(\K (b_{\alpha}))$. Assume $\varphi_i(\K (b_{\alpha})) > \varphi_{i+n'}(\K (b_{\alpha}))$. We can see immediately the condition is equivalent to $\alpha_{i+n'+1} > \alpha_{i+1}$. Then, it satisfies $\tilde{B}_i\K (b_{\alpha}) = \tilde{F}_i\K (b_{\alpha})$. The entries of $\K (b_{\alpha})$ except $i$-th and $(i+1)$-th ones are invariant by the action of $\tilde{B}_i$ and equivalent to the ones of $\K(\tilde{B}_ib_{\alpha})$ in this case. The $i$-th entry is $\alpha_{i+1}+1$ and $(i+1)$-th one is $\alpha_{i+2}+\alpha_{i+n'+2}+\alpha_{i+n'+1}-\alpha_{i+1}-\alpha_{i+1}-\max(\alpha_{i+2},\alpha_{i+n'+2})-1$, and they are equivalent to the $i$-th and $(i+1)$-th one of $\K(\tilde{B}_ib_{\alpha})$ in the condition $\alpha_{i+n'+1} > \alpha_{i+1}$. Similarly, we can see the assertion on the other conditions.

To get the formula for $I(\alpha)$, we compare the dependence of $x$ by applying $\tilde{B}_0$ on both sides of 
$\K(b_\alpha)=x^{I(\alpha)}b_{\alpha'}^\vee$. Noting that $\alpha'_i-\alpha'_{i+n'}=\alpha_{i+n'}-\alpha_i$, we obtain
\begin{align*}
I(\tilde{F}_0\alpha)&=\left\{
\begin{array}{ll}
I(\alpha)+1\quad&(\alpha_n>\alpha_{n'},\alpha_1\ge\alpha_{n'+1}),\\
I(\alpha)+2&(\alpha_n>\alpha_{n'},\alpha_1<\alpha_{n'+1}),
\end{array}
\right. \\
I(\tilde{E}_{n'}\alpha)&=\left\{
\begin{array}{ll}
I(\alpha)\quad&(\alpha_n\le\alpha_{n'},\alpha_1\ge\alpha_{n'+1}),\\
I(\alpha)+1\quad&(\alpha_n\le\alpha_{n'},\alpha_1<\alpha_{n'+1}).
\end{array}
\right.
\end{align*}
Here $\tilde{F}_0\alpha$ or $\tilde{E}_{n'}\alpha$ should be understood as an action on $\B_s(x=1)$. Noting the fact that 
$\tilde{B}_i$ with $i\ne0,n'$ does not change the value of $I(\alpha)$, we obtain the desired formula.
\end{proof}

\begin{lemma}
The $\imath$crystal $\B_s$ is connected.
\end{lemma}
\begin{proof}
Let us introduce the notation $\tilde{B}_i^{\max}$ to mean 
$\tilde{B}_i^{\max}b=\tilde{B}_i^cb$ where $c=\max\{k\ge0\mid \tilde{B}_i^kb\ne0\}$.
Then, for any $b_\alpha\in \B_s$, we have $\tilde{B}_i^{\max}b_\alpha=b_{\alpha'}$ where $\alpha'_{i+n'+1}=0$.
Using this property of $\tilde{B}_i^{\max}$, apply to $b_\alpha$ $\tilde{B}_{n'-1}^{\max},
\tilde{B}_{n'-2}^{\max},\ldots,\tilde{B}_{1}^{\max},\tilde{B}_{0}^{\max},\tilde{B}_{n-1}^{\max},
\ldots,\tilde{B}_{n'+1}^{\max}$, successively. Then, at each application, $\alpha_n,
\alpha_{n-1},\ldots,$ $\alpha_{n'+2},\alpha_{n'+1},\alpha_{n'},\ldots,\alpha_2$ turn out to 
be 0. Since $\tilde{B}_i\tilde{B}_{i+n'}=\mathrm{id}$, we are done.
\end{proof}

\begin{example} \label{ex:reflectionA4}
Set $n=4$. Below is the graphical presentation when we apply the both hand sides of the combinatorial reflection equation 
\eqref{comb RE} on $x^0b_{(3,1,2,1)}\ot y^0b_{(1,2,1,1)}\in\B_7(x)\ot\B_5(y)$.
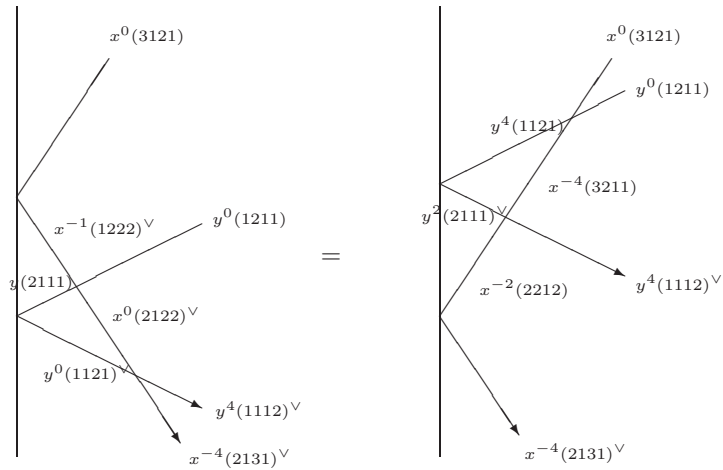
\[
\begin{picture}(300,185)(-32,5)
\put(0,20){
\put(0,-8){\line(0,1){170}}
\put(0,15){\put(0,30){\line(2,1){70}}\put(0,30){\vector(2,-1){70}}}
\put(0,0){\put(0,90){\line(2,3){35}}\put(0,90){\vector(2,-3){62}}}
\put(35,149){\tiny{$x^0(3121)$}}
\put(-3,56){\tiny{$y(2111)$}}
\put(74,80){\tiny{$y^0(1211)$}}
\put(14,76){\tiny{$x^{-1}(1222)^\vee$}}
\put(36,43){\tiny{$x^0(2122)^\vee$}}
\put(10,21){\tiny{$y^0(1121)^\vee$}}
\put(75,7){\tiny{$y^4(1112)^\vee$}}
\put(65,-11){\tiny{$x^{-4}(2131)^\vee$}}
}
\put(115,85){$=$}
\put(160,20){
\put(0,-8){\line(0,1){170}}
\put(0,15){\put(0,30){\line(2,3){65}}\put(0,30){\vector(2,-3){30}}}
\put(0,5){\put(0,90){\line(2,1){70}}\put(0,90){\vector(2,-1){70}}}
\put(63,149){\tiny{$x^0(3121)$}}
\put(74,129){\tiny{$y^0(1211)$}}
\put(19,115){\tiny{$y^4(1121)$}}
\put(41,92){\tiny{$x^{-4}(3211)$}}
\put(-7,81){\tiny{$y^2(2111)^\vee$}}
\put(15,53){\tiny{$x^{-2}(2212)$}}
\put(74,55){\tiny{$y^4(1112)^\vee$}}\put(32,-9){\tiny{$x^{-4}(2131)^\vee$}}
}
\end{picture}
\]
\end{example}

\section*{Acknowledgments}
The authors thank Atsuo Kuniba for letting us know the reference \cite{BPO}, and
Yasuhiko Yamada for interest to our work.
M.O. is supported by JSPS KAKENHI Grant Number~JP19K03426, and
H.W. by JP21J00013.
This work was partly supported by Osaka Central Advanced Mathematical Institute (MEXT Joint Usage/Research Center on Mathematics and Theoretical Physics JPMXP0619217849).

\end{document}